\documentclass[11pt]{amsart}
\usepackage[
backend=biber,
style=alphabetic,
sorting=nyt,
maxbibnames=99
]{biblatex}
\usepackage{tikz}
\usetikzlibrary{fit}
\usepackage[toc]{appendix}

\calclayout

\usepackage{amsmath}
\usepackage{amssymb}
\usepackage{framed,color}
\usepackage{esint}
\usepackage{amsthm}
\usepackage{dsfont}

\usepackage{color}
\usepackage{tkz-euclide}
\usepackage{mathrsfs}
\usepackage{multicol}
\usepackage{stmaryrd}

\usepackage{verbatim}  
\usepackage{hyperref}
\usepackage{setspace}
\usepackage{enumitem}
\usepackage{tikz}
\usetikzlibrary{math}
\usepackage{pgfplots,mathtools}
\usepackage{thmtools}
\usepackage{tkz-euclide}

\usepackage{graphicx}
\usepackage{caption}
\usepackage{subcaption}
\usepackage{fancyhdr}

\usepackage{pgfplots}
\usepgfplotslibrary{polar}
\pgfplotsset{compat=newest}

\usepackage{blindtext}
\usepackage[skip=10pt plus1pt, indent=0pt]{parskip}

\numberwithin{equation}{section}
\theoremstyle{definition}
\newtheorem{definition}{Definition}[section]

\theoremstyle{plain}
\newtheorem*{claim}{Claim}

\newtheorem{lemma}[definition]{Lemma}
\newtheorem{corollary}[definition]{Corollary}
\newtheorem{theorem}[definition]{Theorem}

\newtheorem{proposition}[definition]{Proposition}

\theoremstyle{remark}
\newtheorem{remark}[definition]{Remark}
\newtheorem{example}[definition]{Example}

\DeclareMathOperator{\diam}{diam}

\DeclareMathOperator*{\esssup}{ess\,sup}
\renewcommand{\epsilon}{\varepsilon}

\newcommand{\mbb}{\mathbb}
\newcommand{\mcal}{\mathcal}
\newcommand{\ol}{\overline}

\newcommand{\pde}[2]{\frac{\partial{#1}}{\partial{#2}}}

\newcommand{\de}[2]{\frac{ \mathrm{d} {#1}}{\mathrm{d} {#2}}}
\newcommand{\dde}[2]{\frac{ \mathrm{d}^2{#1}}{\mathrm{d} {#2}^2}}
\newcommand{\df}[1]{\, \mathrm{d}{#1}}

\usepackage{tikz-3dplot}
\usepackage{hyperref}
\usepackage{pgfplots}
\usetikzlibrary{calc,3d,intersections, positioning,intersections,shapes}
\pgfplotsset{compat=1.11} 

\tdplotsetmaincoords{30}{115}%
\tdplotsetrotatedcoords{0}{0}{0}%

\mathtoolsset{showonlyrefs}
\title{Geodesic nets via eigenvalue optimisation}
\author{Duc Hoang Cao}
\date{}                
\address{Department of Mathematics, King's College London, Strand, London, WC2R 2LS, UK}

\email{hoang.cao@kcl.ac.uk}

\begin{document}
\begin{abstract}
We explore a connection between geodesic nets and quantum graphs optimising certain functionals from spectral theory. For surfaces, critical metrics for the normalised $k^{\mathrm{th}}$ eigenvalue of the Laplacian give rise to isometric minimal immersions to a unit sphere. In this spirit we obtain geodesic nets from optimal quantum graphs, and obstructions to the existence of critical metrics.
\end{abstract}

\maketitle

\section{Introduction}
Given a discrete graph $G=(V,E)$, one can construct a quantum---also called metric---graph $(G,g)$ by assigning a length $\ell_e:=\sqrt{g_e}$ on each edge $e\in E$ and viewing $e$ as an interval $[0,\ell_e]$. We define the Laplacian on $G$ as $\Delta_g:=-\df{}^2/\df {x}^2$ and consider the eigenvalue problem with Kirchhoff--Neumann vertex conditions, 
\begin{equation}\label{classical_eigenvalue_problem}
	\begin{cases}
		\Delta_g f = \lambda f &\text{on edges};\\
		f \text{ is continuous} & \text{on } (G,g); \\
		\sum_{e \in E_v} f_e'(v) = 0 &\text{at } v \in V,
    \end{cases}
\end{equation}
where $E_v$ is the set of edges attached at $v$ and $f'_e(v)$ is the derivative of $f$ at $v$ in the direction pointing out of $v$ into the edge $e\in E_v$. If $(G,g)$ is a compact connected quantum graph, the eigenvalues are discrete and form an increasing sequence
\begin{equation}
    0=\lambda_0(G,g)<\lambda_1(G,g)\le \lambda_2(G,g)\le \cdots\le \lambda_n(G,g)\le \cdots\nearrow\infty.
\end{equation}

We define the scale-invariant normalised eigenvalues
\begin{equation}
    \ol{\lambda_k}(g):=\lambda_k(g)L(g)^2,\quad \forall k\ge 0,
\end{equation}
where $L(g)$ is the total length of $(G,g)$.

The main results of this paper are motivated by the analogous problem for surfaces: given a compact surface $M$, Nadirashvili \cite{Nadirashvili1996} and El Soufi--Ilias \cite{ELSOUFI_lambda1} proved that the maximising metric of the first positive normalised eigenvalue, $\lambda_1(M,g)\hbox{Area}(M,g)$, is induced by minimal isometric immersions into spheres, where $\lambda_1(M,g)$ is the first positive Laplacian eigenvalue, a result which was extended to higher eigenvalues in \cite{ELSOUFI200889}. Later, Karpukhin--Métras \cite{relaxing} extended these results for extremal metrics of normalised eigenvalue on closed manifolds of dimension at least three. We investigated the one-dimensional analogue in this paper.

Specifically, we are interested in connections between quantum graphs and geodesic nets. A geodesic net is a collection of points, $\mathcal{V}$, in a sphere together with a collection of geodesics, $\Gamma$, whose endpoints are in $\mathcal{V}$ and such that at every point in $\mathcal{V}$, the sum of outward unit tangent vectors of geodesics attached to that point vanishes. By Takahashi's Theorem, \cite[Theorem 3]{Takahashi}, the coordinate functions of a geodesic net are eigenfunctions of the Laplacian on the induced quantum graph. Our main theorem shows that the metric on this graph is an extremal metric for the normalised eigenvalues functional.

\begin{theorem}\label{main_theorem1}
	Let $G$ be a finite compact connected quantum graph and assume that $g_*$ is an extremal metric for the functional $\ol{\lambda_k}(g):=\lambda_k(g)L(g)^2$. Then there exist some $\lambda_k(g_*)$-eigenfunctions $f_1,f_2,\dots,f_n$ such that
	\begin{equation}\label{eq_main_theorem1}
		\sum_{i=1}^n \left[|\nabla_{g_*}f_i|_{g_*}^2 +\lambda_k(g_*) f_i^2\right]=1.
	\end{equation}
	Conversely, if there exist $\lambda_k(g)$-eigenfunctions $f_1,\dots,f_n$ satisfying  \eqref{eq_main_theorem1} and, additionally, $\lambda_k(g)>\lambda_{k-1}(g)$ or $\lambda_{k}(g)<\lambda_{k+1}(g)$, then $g$ is extremal for the functional $\ol{\lambda_k}(\cdot)$.
\end{theorem}

Unlike higher-dimensional analogues, extremal metrics of graphs may not always be induced by isometric minimal immersions into spheres. To rectify this, we follow ideas from \cite{relaxing} by adding a density function to the eigenvalue problem and consider the extremal problem in the space of metrics and smooth density functions. Although the number of extremal points decreases after adding density functions, these extremal points satisfy stronger conditions, so that we expect more properties from the eigenfunctions.

Let $\rho:(G,g)\to\mbb{R}_+$ be smooth on edges, and consider the eigenvalue problem
\begin{equation}\label{density_eigenvalue_problem}
	\begin{cases}
		\Delta_g f = \lambda \rho f &\text{on edges};\\
		f \text{ is continuous} & \text{on } (G,g); \\
		\sum_{e \in E_v} f_e'(v) = 0 &\text{at } v \in V.
    \end{cases}
\end{equation}

If $(G,g)$ is compact and connected, the eigenvalues of \eqref{density_eigenvalue_problem} are discrete and also form a sequence
\begin{equation}
    0=\lambda_0(g,\rho)<\lambda_1(g,\rho)\le \lambda_2(g,\rho)\le \cdots\le \lambda_n(g,\rho)\le \cdots\nearrow\infty.
\end{equation}

We define the naturally normalised eigenvalues as follows:
\begin{equation}\label{natural_normalisation}
    \ol{\lambda_k}(g,\rho):=\lambda_k(g,\rho)L(g)\int_{(G,g)}\rho\df x,\quad \forall k\ge 0.
\end{equation}

Then, $\ol{\lambda_k}(g,\rho)$ is invariant under rescaling. For a discussion of the naturality of this normalisation, see \cite[Section 4]{MR4311579}. By considering density functions, we improve Theorem \ref{main_theorem1} as follows.

\begin{theorem}\label{main_theorem2}
	Let $G$ be a finite compact connected graph and suppose that $(g_*,\rho_*)$ is an extremal pair for the functional $\ol{\lambda_k}(g,\rho)$. Then $\rho_*$ is a constant function and there exist some $\lambda_k(g_*,\rho_*)$-eigenfunctions $f_1,f_2,\dots,f_n$ such that
	\begin{equation}\label{eq_main_2}
		\sum_{i=1}^n|\nabla_{g_*}f_i|_{g_*}^2=1,\quad \lambda_k(g_*,\rho_*)\sum_{i=1}^n f_i^2=\frac{L(g_*)}{\int_{(G,g_*)} \rho_* \df x}.
	\end{equation}
	Conversely, if there exist $\lambda_k(g,\rho)$-eigenfunctions $f_1,\dots,f_n$ satisfying \eqref{eq_main_2} and, additionally, $\lambda_k(g,\rho)>\lambda_{k-1}(g,\rho)$ or $\lambda_{k}(g,\rho)<\lambda_{k+1}(g,\rho)$, then $(g,\rho)$ is extremal for $\ol{\lambda_k}(\cdot,\cdot)$.
\end{theorem}

The proof is a one-dimensional analogue of \cite[Theorem 8]{relaxing}. As a direct consequence of Theorem \ref{main_theorem2}, the image of the map $(f_1,f_2,\dots, f_n):(G,g_*)\to\mbb{R}^n$ is a geodesic net on a sphere.

Now that characterisation is settled, we turn to the existence of extremal pairs for the normalised eigenvalue functional. We note that bounds for $\ol{\lambda_k}(g)$ have been obtained in \cite{ariturk2016, Kennedy16, Berkolaiko_2017}. In \cite{Band_optimiser}, Band--L\'evy presented maximising and minimising metrics for $\ol{\lambda_1}(g)$ for certain graph topologies. In this paper, we are interested in bounds for $\ol{\lambda_k}^{(\alpha)}(\cdot,\cdot)$. In Section \ref{eigenvalue_bounds}, we prove that for every connected graph $G$, we have
\begin{equation}\label{no_maximiser_0}
    \sup_{(g,\rho)}\ol{\lambda_1}(g,\rho)=\infty,
\end{equation}
and
\begin{equation}\label{lower_bound}
    \inf_{(g,\rho)}\ol{\lambda_1}(g,\rho)\ge 1.
\end{equation}

Therefore, there is no maximiser for the functional $\ol{\lambda_1}(\cdot,\cdot)$. Although we do not know about the existence of the minimiser, one can apply Theorem \ref{main_theorem2} to see that if $(g,\rho)$ is a minimiser for the functional $\ol{\lambda_k}(\cdot,\cdot)$, then $g$ must be an minimiser for the functional $\ol{\lambda_k}(\cdot)$, and we can construct an isometric minimal immersion from $(G,g)$ to some spheres via $\lambda_k(g)$-eigenfunctions. We use these facts to show that on necklace graphs, there is no minimiser for the functional $\ol{\lambda_1}(\cdot,\cdot)$.

We investigate other normalisations for eigenvalues. Consider the family of normalisations
\begin{equation}\label{generalisation}
	\ol{\lambda_k}^{(\alpha)}(g,\rho):=\lambda_k(g,\rho)\left(\sum_{e\in E}g_e^{1-\alpha}\right)\left(\sum_{e\in E}g_e^\alpha\int_0^{\ell_e}\rho_e \df x\right),
	\end{equation}
for some $\alpha\in\mbb{R}$. The natural normalisation corresponds to $\alpha=1/2$. For convenience, we denote:
\begin{equation}
    F_\alpha(g):=\sum_{e\in E}g_e^{1-\alpha},\quad \text{and}\quad H_\alpha(g,\rho):=\sum_{e\in E}g_e^\alpha\int_0^{\ell_e}\rho_e \df x.
\end{equation}

\begin{theorem}\label{main_theorem3} Let $G$ be a finite compact connected graph and $\alpha\in\mbb{R}\backslash\{1/2,1\}$. Suppose that $(g_*,\rho_*)$ is an extremal pair for $\ol{\lambda_k}^{(\alpha)}(g,\rho)$. Then, $g_*$ is a regular metric, i.e. all lengths of $(G,g_*)$ are the same, $\rho_*$ is a constant function and there exist $\lambda_k(g_*,\rho_*)$-eigenfunctions $f_1,f_2,\dots,f_n$ such that
	\begin{equation}\label{natural_generalisation_result_1}
		\sum_{i=1}^n|\nabla_{g_*}f_i|_{g_*}^2=1,
	\end{equation}
	and
	\begin{equation}\label{natural_generalisation_result_2}
		\lambda_k(g_*,\rho_*)\sum_{i=1}^n f_i^2 =\frac{ g_*^{\alpha-1/2}F_\alpha(g_*)}{2(1-\alpha)g_*^{-\alpha+1/2}H_{\alpha}(g_*,\rho_*)+(2\alpha-1) \rho_* g_*^{\alpha-1/2}F_\alpha(g_*)}.
	\end{equation}

Conversely, if there exist $\lambda_k(g,\rho)$-eigenfunctions $f_1,\dots,f_n$ satisfying \eqref{natural_generalisation_result_1} and \eqref{natural_generalisation_result_2} and, additionally, $\lambda_k(g,\rho)>\lambda_{k-1}(g,\rho)$ or $\lambda_{k}(g,\rho)<\lambda_{k+1}(g,\rho)$, then $(g,\rho)$ is extremal for the functional $\ol{\lambda_k}^{(\alpha)}(\cdot,\cdot)$.
\end{theorem}

Notice that from Theorem \ref{main_theorem2} and Theorem \ref{main_theorem3}, an extremal pair of metric and density function for $\ol{\lambda_k}^{(\alpha)}(\cdot,\cdot)$ implies an isometric minimal immersion into spheres if $\alpha\ne 1$.

\subsection{Outline of the paper}
We first recall some basic definitions and notations on quantum graphs in Section \ref{set_up}. In Section \ref{proof_section}, we present a general result, which implies our main theorems directly. The proof follows the structure of \cite{relaxing}. We show a connection between Theorem \ref{main_theorem1} and Theorem \ref{main_theorem2} by analysing the spectrum of $\Delta_g$ with an extremal metric $g$ for $\ol{\lambda_k}(\cdot)$. In particular, we prove that $(g,1)$ is an extremal pair of the functional $\ol{\lambda_k}(\cdot,\cdot)$ if $4\lambda_k(g)$ is not an eigenvalue of the eigenvalue problem $\Delta_g f=\lambda f$. We then extend the results for extremal pairs for $\ol{\lambda_k}^{(\alpha)}(\cdot,\cdot)$ as given in equation \eqref{generalisation}. In Section \ref{eigenvalue_bounds}, we prove  \eqref{no_maximiser_0} and \eqref{lower_bound}. In the final section, we look at the spectrum of $\Delta_g$, where $g$ is a maximiser/minimiser of $\ol{\lambda_1}(\cdot)$ on flower, mandarin and necklace graphs. We try to construct a map from these quantum graphs to spheres.
\subsection{Acknowledgements}
The main results of this paper are part of the author's master's (at University College London) and PhD (at King's College London) projects, under the supervision of Mikhail Karpukhin and Jean Lagac\'e. The author would like to thank his advisors for their guidance and feedback.
\section{Set up and notations}\label{set_up}
\subsection{Sobolev spaces on quantum graphs}
Through this paper, we fix a finite connected discrete graph, called $G$, and vary metrics, so that it is convenient to fix a universal metric on $G$, called $g_0$, and view other metrics on $G$ as functions on $(G,g_0)$. We define $g_0$ as the metric on $G$ such that all edges of $G$ have length one, and $G_0$ to be the metric graph $(G,g_0)$.

Note that for smooth regular curves, the curve length is the only intrinsic geometric property since we can always parametrise curves by arc length. Hence, for any metric $g$ on $G$, we can view it as a vector in $\mbb{R}^{|E|}_+$ and on edges, the metric can be written as:
\begin{equation}
    \df x_g^2=g_e\df x^2=\ell^2_e\df x^2,\quad \forall e\in E.
\end{equation}

Then, we can consider any function $f:(G,g)\to\mbb{R}$ as a function $f:G_0\to\mbb{R}$. We define the $L^p$ space of functions on $G$ as follows: 
\begin{equation}
    L^p(G,g):=\left\{f:(G,g)\to\mbb{R}\mid f_e\in L^p(0,\ell_e)\;\forall e\in E\right\},
\end{equation}
for all $p\in [1,\infty]$, with the induced norm:
\begin{equation}
    \|f\|_{L^p(G,g)}:=\begin{cases}\left(\sum_{e\in E}\|f_e\|_{L^p(0,\ell_e)}^p\right)^{1/p}, &\text{if }p<\infty;\\
    \max_{e\in E}\|f_e\|_{L^\infty(0,\ell_e)},&\text{if }p=\infty,
    \end{cases}
\end{equation}
and
\begin{equation}
    W^{k,p}(G,g):=\left\{f:(G,g)\to\mbb{R}\mid f\text{ is continuous},\; f_e\in W^{k,p}(0,\ell_e)\,\forall {e\in E}\right\},
\end{equation}
for all $k\in\mbb{N}\cup\{0\}$ and $p\in (1,\infty]$, with the induced norm:
\begin{equation}
    \|f\|_{W^{k,p}(G,g)}:=\begin{cases}
    \left(\sum_{e\in E}\|f_e\|_{W^{k,p}(0,\ell_e)}^p\right)^{1/p},&\text{if }p<\infty;\\
    \max_{e\in E}\|f_e\|_{W^{k,\infty}(0,\ell_e)},&\text{if }p=\infty.
    \end{cases}
\end{equation}

For all $k\in\mbb{N}\cup\{\infty\}$, we define $C^k(G,g)$ to be the space of continuous functions on $(G,g)$ such that their restrictions on edges are $C^k$ functions. Observe that for all metric $g$ in $G$, the space $L^p(G,g)$ is equivalent to $L^p(G_0)$, so that we can refer all $L^p$ space in $G$ as $L^p(G)=L^p(G_0)$. We induce the same notation for $W^{k,p}(G)$ and $C^k(G)$.

Finally, for a function $f(t)$, we denote
\begin{equation}
    \dot f(t_0):=\left.\frac{\df f}{\df t}\right|_{t=t_0},
\end{equation}
whenever the derivative exists.
\subsection{Laplacian eigenvalues}
For a metric $g$ on $G$, we define the gradient to be:
\begin{equation}
    \nabla_g f:=\de{f}{x_g}\cdot \de{}{x_g}=\frac{1}{g_e}\cdot \de{f}{x}\cdot\de{}{x},
\end{equation}
for all $f\in W^{1,2}(G)$ and define the Laplacian to be:
\begin{equation}
    \Delta_g f:=-\dde{f}{x_g}=-\frac{1}{g_e}\cdot\dde{f}{x},
\end{equation}
for all $f\in W^{2,2}(G)$. For the eigenvalue problem \eqref{classical_eigenvalue_problem} on compact quantum graph $(G,g)$, we can compute the $k$-th eigenvalue using the variational characterisation as follows
\begin{equation}
    \lambda_k(g)=\min_{\substack{S\subset W^{1,2}(G)\\ \dim S=k+1}}\max_{\substack{f\in S\backslash\{0\}}} R(G,g;f),\quad R(G,g;f):=\frac{\int_{G_0} |\nabla_g f|_g^2 \df x_g}{\int_{G_0} f^2 \df x_g}=\frac{\int_{G_0} \frac{(f')^2}{\sqrt{g}} \df x}{\int_{G_0} f^2\sqrt{g} \df x}.
\end{equation}

For the eigenvalue problem with density functions as given in \eqref{density_eigenvalue_problem}, the eigenvalues also obey the variational characterisation
\begin{equation}
\lambda_k(g,\rho)=\min_{\substack{S\subset W^{1,2}(G)\\ \dim S=k+1}}\max_{\substack{f\in S\backslash\{0\}}} R(G,g,\rho;f),\quad R(G,g,\rho;f):=\frac{\int_{G_0} |\nabla_g f|_g^2 \df x_g}{\int_{G_0} f^2\rho \df x_g}=\frac{\int_{G_0} \frac{(f')^2}{\sqrt{g}} \df x}{\int_{G_0} f^2\rho\sqrt{g} \df x}.
\end{equation}

By abusing notations, we simply write $R(f)$ for the Rayleigh quotient, if the graph, the metric and the density function are fixed.
\subsection{Space of density functions and metrics}
	For a quantum graph $G$, we denote the space of density functions as follows
	\begin{equation}
	    \mathcal{C}_+(G):=\{f:G_0\to \mbb{R}_+\mid  f_e\in C^\infty(0,1)\quad\forall e\in E\}.
	\end{equation}
	
	We denote $\mathcal{A}_+(G):=\mathbb{R}_{+}^{|E|}\times \mathcal{C}_+(G)$ to be the space of pairs of metrics and density functions on $G$.
    \begin{definition}
    For a functional $F:\mathcal{A}_+(G)\to\mbb{R}$, we say that a pair $(g_*,\rho_*)$ is $F$-extremal if for all one-parameter smooth family of pairs $\{(g(t),\rho(t))\}\subset \mathcal{A}_+(G)$ such that $g(0)=g_*,\rho(0)=\rho_*$, one has either
	\begin{equation}F(g(t),\rho(t))\le F(g_*,\rho_*)+o(t)\quad\text{or}\quad F(g(t),\rho(t))\ge F(g_*,\rho_*)+o(t),
    \end{equation}
	as $t\to 0$. 
    \end{definition}

\begin{definition}
Let $\mathcal{A}(G):=\mbb{R}^{|E|}\times\mathcal{C}(G)$, where
	\begin{equation}\label{smooth_space}
    \mathcal{C}(G):=\{f:G_0\to \mbb{R}\mid  f_e\in C^\infty(0,1)\quad\forall e\in E\}.
    \end{equation}
	
	For $(g_*,\rho_*)\in\mathcal{A}_+(G)$, we define the tangent plane of $(g_*,\rho_*)$ in $\mathcal{A}(G)$, denoted as $T\mathcal{A}_{(g_*,\rho_*)}(G)$, as the closure of the set
	\begin{equation}\left\{(\phi,\eta)\in {\mathcal{A}}(G):\exists \{(g(t),\rho(t))\}\subset\mathcal{A}_+(G), (g(0),\rho(0))=(g_*,\rho_*),(\dot g(0),\dot \rho(0))=(\phi,\eta)\right\},\end{equation}
	in the space $\mathcal{H}(g_*):=L^2(G,g_*)\times L^2(G,g_*)$, with the induced inner product
	\begin{equation}\langle(f_1,f_2),(h_1,h_2)\rangle_{\mathcal{H}(g_*)}=\int_{G_0} f_1h_1+f_2h_2 \df x_{g_*},
    \end{equation}
	for all $f_1,f_2,h_1,h_2\in L^2(G,g_*).$ Moreover, for a Lipschitz functional $F:\mathcal{A}(G)\to\mbb{R}$, we define the gradient of $F$ at $(x_0,y_0)$, $\nabla F(x_0,y_0)\in T\mathcal{A}_{(x_0,y_0)}(G)$, as 
	\begin{equation}\left.\de{}{t} F(x(t),y(t))\right|_{t=0}=\left\langle \nabla F(x_0,y_0),(\dot x(0),\dot y(0))\right\rangle_{\mathcal{H}(g_*)},\end{equation}
	for all $\{(x(t),y(t))\}\subset\mathcal{H}(g_*)$ such that $x(0)=x_0,y(0)=y_0$, if the derivative exists.
\end{definition}

 The functionals that we study in this paper are normalised eigenvalue functionals, i.e.
    \begin{equation}
        F(g,\rho)=\lambda_k(g,\rho)N(g,\rho),
    \end{equation}
where $N:\mathcal{A}_+(G)\to(0,\infty)$ is smooth and satisfies the homogeneity condition,
\begin{equation}\label{normalising_condition}
N(a g,b \rho)=ab N(g,\rho).
\end{equation}
for all $(g,\rho)\in\mathcal{A}_+(G)$ and $a,b>0$.

\section{Extremal pairs for normalised eigenvalues}\label{proof_section}
In this section, we prove the following generalisation of our main theorems:

\begin{theorem}\label{general_theorem}
Let $G$ be a connected finite graph and $k\in\mbb{N}$. Suppose that $(g_*,\rho_*)\in\mathcal{A_+}(G)$ is an extremal pair of the functional
\begin{equation}
    F(g,\rho):=\lambda_k(g,\rho)N(g,\rho),
    \end{equation}
    where $N$ is a smooth functional and $F$ is invariant under scaling. Let $\nabla N(g_*,\rho_*)=(N_1,N_2)$, then there exists $f_1,f_2,\dots,f_n\in E(\lambda_k(g_*,\rho_*))$ such that:
	\begin{equation}\label{general_main_equation}
    \sum_{i=1}^n|\nabla_g f_i|^2=2gN_1-\rho N_2,\quad \sum_{i=1}^n f_i^2=\frac{N_2}{\lambda_k(g_*,\rho_*)}.\end{equation}
    
    Conversely, if there exists $\lambda_k(g,\rho)$-eigenfunctions $f_1,\dots,f_n$ satisfying the equation \eqref{general_main_equation} and, additionally, $\lambda_k(g,\rho)>\lambda_{k-1}(g,\rho)$ or $\lambda_{k}(g,\rho)<\lambda_{k+1}(g,\rho)$, then $(g,\rho)$ is an extremal pair of $F$.
\end{theorem}

The proof follows the structure from \cite{surface_case,ELSOUFI200889,relaxing}. First, we prove that for every smooth one-parameter family of pairs $(g(t),\rho(t))$ in $\mathcal{A}_+(G)$, the functional $\lambda_k(t):=\lambda_k(g(t),\rho(t))$ is Lipschitz for small $t$ and the derivative of $\lambda_k$, when it exists, can be written in the form:
\begin{equation}
    \nabla \lambda_k(g,\rho)=Q(u;g,\rho):=-\left(\frac{1}{2g}\left(|\nabla_g u|^2 +\lambda_k(g,\rho)u^2\rho\right),\lambda_k(g,\rho) u^2\right),
    \end{equation}
	for some $u\in E(\lambda_k(g,\rho))$. Then, we show that if $(\phi,\eta)\in\mathcal{H}(g_*)=L^2(G,g_*)\times L^2(G,g_*)$ satisfies
    \begin{equation}
        \langle \nabla N(g_*,\rho_*),(\phi,\eta)\rangle_{\mathcal{H}(g_*)}=0,
    \end{equation}
    then there exists $u\in E(\lambda_k(g_*,\rho_*))\backslash\{0\}$ such that
    \begin{equation}
        \langle Q(u;g_*,\rho_*),(\phi,\eta)\rangle_{\mathcal{H}(g_*)}=0.
    \end{equation}

    We use these facts, together with the Hahn--Banach theorem, to show that $\nabla N(g_*,\rho_*)$ is in the convex hull of the set
    \begin{equation}
        \{-Q(u;g_*,\rho_*):u\in E(\lambda_k(g_*,\rho_*)\}
    \end{equation}
    in the space $\mathcal{H}(g_*)$.

\subsection{Lipschitz continuity of eigenvalue functionals}
To prove Lipschitz continuity of eigenvalue functionals, we use the variational characterisation and the following lemma: 

	\begin{lemma}\label{lemma_lipschitz}
    Let $G$ be a finite compact graph, $(g,\rho)\in\mathcal{A}_+(G)$ and $I\subset\mbb{R}$ be a closed interval. Let $k\in \mbb{R}$ and consider a functional $J(t,u):I\times E(\lambda_k(g,\rho))\to\mathcal{C}(G)$ smooth in $t$ and such that there exists a constant $C>0$ so that
    \begin{equation}\label{condition_lipschitz}
		\left|\pde{}{t}J(t,u(x))\right|\le C\left[u(x)^2+u'(x)^2+1 \right],
	\end{equation}
	for all $t\in I, u\in E(\lambda_k(g,\rho))$, and  $x\in G_0$. Then, for all $u\in E(\lambda_k(g,\rho))$ such that $\int_{G_0} u^2\rho \df x_{g}=1$, the function $Z:I\to\mbb{R}$ given by
    \begin{equation}
        Z(t):=\int_{G_0} J(t,u(x)) \df x
    \end{equation}
	is Lipschitz continuous in $I$ with constant
    \begin{equation}
        C'=C\left(\frac{1}{m}+M\lambda_k(g,\rho)+|E|\right),
    \end{equation}
    where 
    \begin{equation}
         m:=\min\{\rho_e(x)\sqrt{g_e}\mid  e\in E, x\in[0,1]\},\quad\text{and}\quad M:=\max\{\sqrt{g_e}: e\in E\}.
    \end{equation}
\end{lemma}

\begin{proof}
	For all $s_1,s_2\in I$ and $x\in G_0$, the mean value theorem implies
	\begin{equation}
    \begin{aligned}
		|J(s_1,u(x))-J(s_2,u(x))|&\le|s_1-s_2|\cdot \max\left\{\left|\pde{}{t}J(t,u(x))\right|:t\in I\right\}\\
		&\le C|s_1-s_2|\cdot \left[u(x)^2+u'(x)^2+1 \right].
        \end{aligned}
	\end{equation}
	
    Hence, we have:
	\begin{equation}\label{Lipschitz_1}
    \begin{aligned}
		|Z(s_1)-Z(s_2)|\le C|s_1-s_2|\cdot\left(\|u\|_{W^{1,2}(G_0)}^2+|E|\right),\quad \forall s_1,s_2\in I.
        \end{aligned}
	\end{equation}
	Observe that:
	\begin{equation}\label{Lipschitz_2}
    \|u\|^2_{L^2(G_0)}=\int_{G_0} u^2\df x\le\int_{G_0} u^2\frac{\rho\sqrt{g}}{m}\df x= \frac{1}{m}\int_{G_0} u^2\rho \df x_{g}=\frac{1}{m}.
    \end{equation}
	
    To bound $\|u'\|_{L^2(G_0)}$, we use the fact $u\in E(\lambda_k(g,\rho))$ to get:
	\begin{equation}\label{Lipschitz_3}
    \begin{aligned}
		\|u'\|^2_{L^2(G_0)}\le M\int_{G_0} \frac{(u')^2}{\sqrt{g}}\df x= M\int_{G_0}|\nabla_{g} u|_g^2 \df x_{g}=M\lambda_k(g,\rho),
        \end{aligned}
	\end{equation} 
	
    Our assertions then follow from inequalities \eqref{Lipschitz_1}, \eqref{Lipschitz_2} and \eqref{Lipschitz_3}.
\end{proof}

\begin{lemma}\label{derivative_lambda}
	Let $\left(g(t),\rho(t)\right)$ be a smooth one-parameter family of pairs in $\mathcal{A}_+(G)$ such that $g(0)=g_*$, and $\rho(0)=\rho_*$. Fix $k\in\mbb{N}$ and we consider the functional $\lambda_k(t):=\lambda_k\left(\rho(t),g(t)\right)$. Then, for almost every $t_0$ close to 0, there exists $\epsilon>0$ such that $\lambda_k(t)$ is Lipschitz on $[t_0-\epsilon,t_0+\epsilon]$. Moreover, if $\dot\lambda_k(t_0)$ exists, then:
	\begin{equation}
	    \dot{\lambda}_k(t_0) =-\int _{G_0}\frac{\dot g(t_0)}{2g(t_0)}\left[ |\nabla_{g(t_0)} u| ^{2}+\lambda_k(t_0)u^{2}\rho (t_0)\right] +\lambda _{k}(t_0)u^{2} \dot \rho(t_0)\df x_{g(t_0)},
	\end{equation}
	for all $u\in E(\lambda_k(t_0))$ such that $\int_{G} u^2\rho(t_0) \df x_{g(t_0)}=1$. Consequently, if $\nabla \lambda_k(g,\rho)$ exists, then
	\begin{equation}
    \nabla \lambda_k(g,\rho)=Q(u;g,\rho):=-\left(\frac{1}{2g}\left(|\nabla_g u|^2 +\lambda_k(g,\rho)u^2\rho\right),\lambda_k(g,\rho) u^2\right),
    \end{equation}
	for all $u\in E(\lambda_k(g,\rho))$ such that $\int_{G_0} u^2\rho \df x_{g}=1$.
\end{lemma}

\begin{proof}
	For almost every small $t_0$, there is a neighbourhood around $t_0$ on which the multiplicity of $\lambda_k(t)$ is constant. Fix such $t_0$ and suppose that $I=[t_0-\epsilon,t_0+\epsilon]$ is that neighbourhood around $t_0$. Let $h\in\mathbb{N}$ be such that $\lambda_k(t)=\lambda_h(t)>\lambda_{h-1}(t)$. We denote $\mathcal{E}(t):=\bigoplus_{j=0}^{h-1}E(\lambda_j(t))$ and consider the orthonormal projection $P_t:L^2(G,g(t))\to \mathcal{E}(t)$. Then, $\dim\mathcal{E}(t)=h$ and $\mathcal{E}(t)$ varies smoothly for small $t$. Let $t_1,t_2\in I$ and without loss of generality, suppose that $\lambda_k(t_1)\le \lambda_k(t_2)$, and let $u$ be an eigenfunction of $\lambda_k(t_1)$ such that $\int_{G_0} u^2\rho(t_1)\df x_{g(t_1)}=1$. The idea to estimate $\lambda_k(t_2)-\lambda_k(t_1)$ is to use the variational characterisation:
		\begin{equation}\label{estimate_Lipschitz}
			\lambda_k(t_2)-\lambda_k(t_1)\le \frac{\int_{G_0} |\nabla_{g(t_2)} (u-P_{t_2}u)|_{g(t_2)}^2 \df x_{g(t_2)}}{\int_{G_0}  (u-P_{t_2}u)^2\rho(t_2) \df x_{g(t_2)}}-\int_{G_0} |\nabla_{g(t_1)}u|_{g(t_1)}^2 \df x_{g(t_1)}=\frac{H(t_2)}{D(t_2)}-\frac{H(t_1)}{D(t_1)},
		\end{equation}
	where 
		\begin{equation}\label{def_D_H}
			D(t):=\int_{G_0}(u-P_tu)^2\rho(t) \df x_{g(t)},\quad\text{and}\quad H(t):=\int_{G_0} |\nabla_{g(t)} (u-P_tu)|_{g(t)}^2\df x_{g(t)},
		\end{equation}
	and to prove that $D$ and $H$ are Lipschitz continuous in $I$. 
	
	\begin{claim}
		There exists $C=C(I)$ which does not depend on the choice of $u$, $t_1$ and $t_2$ such that the functions $D(t),H(t)$ as given in \eqref{def_D_H}	are Lipschitz with constant $C$ in $I$.
	\end{claim}
	
	\begin{proof}[Proof of claim]
		Consider 
		\begin{equation}J_1(t,u(x)):=[u(x)-P_tu(x)]^2 \rho(t)\sqrt{g(t)},\end{equation}
		and
		\begin{equation}
        J_2(t,u(x)):=|\nabla_{g(t)}(u(x)-P_tu(x))|^2\sqrt{g(t)}=\frac{1}{\sqrt{g(t)}}\left[\pde{}{x}\left(u-P_tu\right)\right]^2.
        \end{equation}
		
		We prove that $J_1$ and $J_2$ satisfy the hypotheses \eqref{condition_lipschitz}. Let $G_I:=I\times G_0$ and consider the following norm:
		\begin{equation}
		    \|f\|_{W^{2,\infty}(G_I)}:=\max_{e\in E}\|f_e\|_{W^{2,\infty}(I\times(0,1))},
		\end{equation}
		for all function $f:G_I\to\mbb{R}$ such that $f_e\in W^{2,\infty}(I\times (0,1))$ for all $e\in E$ and $t\in I$. We show that there exists a constant $C_0=C_0(I,h)>0$ such that
			\begin{equation}\|P_tu\|_{W^{2,\infty}(G_I)}\le C_0,\quad \forall u\in E(\lambda_k(t_1)).\end{equation}
        
        Indeed, let $f^j_t: G_I\to\mbb{R}$ be such that $\{f^0_t,f^1_t,\dots,f^{h-1}_t\}$ is an orthonormal basis for $\mcal{E}(t)$. Without loss of generality, we assume that $f^j_t$ is also smooth in $t$ for all $j$. Then, there exists a constant $C_1=C_1(I,h)$ such that $C_1\ge \|f^i_t\|_{W^{2,\infty}(G_I)}$ for all $i\in\{1,2,\dots, h-1\}$. Moreover, using the Cauchy--Schwarz inequality, we have:
			\begin{equation}\begin{aligned}
				|(u,f^j _{t})_{L^2(G,g(t))}|^2\le \|u\|^2_{L^2(G,g(t))}\le \frac{M}{m}\int_{G_0} u^2\rho(t_1) \df x_{g(t_1)}=\frac{M}{m},\quad \forall j\le h-1,
			\end{aligned}
            \end{equation}
			where
			\begin{equation}M=\max \left\{\sqrt{g_e(t)}:t\in I, e\in E\right\},\quad m=\min \left\{\rho_e(t)(x)\sqrt{g_e(t)}:t\in I, e\in E,x\in [0,1]\right\}.
            \end{equation}
			
			Since $I$ is compact, we have $m>0$ and $M<\infty$. Therefore,
			\begin{equation}
				\|P_tu\|_{W^{2,\infty}(G_I)}\le\sum_{i=0}^{h-1}|(u,f_t^i)_{L^2(G,g(t))}|\cdot \| f_t^i\|_{W^{2,\infty}(G_I)}\le C_1h\sqrt{\frac{M}{m}}.
			\end{equation}

		Hence, applying the Cauchy--Schwarz inequality gives us:
		\begin{equation}
        \begin{aligned}
			\left|\pde{J_1}{t}(t,u(x))\right|&\le\left\|\rho(t)\sqrt{g(t)}\right\|_{W^{2,\infty}(G_I)}\left((u-P_tu)^2+2\|P_tu\|_{W^{2,\infty}(G_I)}|u-P_tu|\right)\\
            &\le C_2(u^2+1),
            \end{aligned}
		\end{equation}
		and
		\begin{equation}
        \begin{aligned}
			\left|\pde{J_2}{t}(t,u(x))\right|&\le \left\| \frac{1}{\sqrt{g(t)}} \right\|_{W^{2,\infty}(G_I)}\cdot \left(\left(\pde{}{x}(u-P_tu)\right)^2+\|P_tu\|_{W^{2,\infty}(G_I)}\left|\pde{}{x}(u-P_tu)\right|\right)\\
            &\le C_2((u')^2+1),
            \end{aligned}
		\end{equation}
		for all $(x,t)\in G_I$, where $C_2=C_2(I)>0$. By Lemma \ref{lemma_lipschitz}, we have that $D(t)$ and $H(t)$ are Lipschitz with some constant:
        \begin{equation}
            C=C_2\left(\frac{1}{m}+M\Lambda+|E|\right)
        \end{equation}
        in $I$, where $\Lambda=\max_{t\in I}\lambda_k(t)$ (since $I$ is compact, $0<C<\infty$).
        	\end{proof}
	
	Let us return to inequality \eqref{estimate_Lipschitz}. To use $u-P_{t_2}u$ as a trial function to get an upper bound for $\lambda_k(t_2)$, we need $u-P_{t_2}u\not\equiv 0$. Setting 
    \begin{equation}
    W(t):=\int_{G_0} (u-P_tu)^2\rho(t) \df x_{g(t)},\quad \forall t\in I.
    \end{equation} 
    
    We see that $W(t)$ is Lipschitz in $I$ with a constant $C_3=C_3(I)$ (the proof is similar to the proof of Lipschitz continuity of $D(t)$). Hence,
	\begin{equation}
    W(t)\ge W(t_1)-C_3|t-t_1|=\int_{G_0} u^2\rho(t_1) \df x_{g(t_1)}-C_3|t-t_1|=1-C_3|t-t_1|,\quad\forall t\in I.
    \end{equation}	
	
	Note that $C_3$ is independent of $t_1$ and $u$, so that if we
	choose a smaller closed interval $\tilde{I}\subset I$, then $u\not\equiv P_tu$ for all $t\in\tilde{I}$. By abusing notations, we take $I=\tilde{I}$, then:
	\begin{equation}
    \lambda_k(t_2)-\lambda_k(t_1)\le \frac{\int_{G_0} |\nabla_{g(t_2)} (u-P_{t_2}u)|_{g(t_2)}^2 \df x_{g(t_2)}}{\int_{G_0}  (u-P_{t_2}u)^2 \df x_{g(t_2)}}-\int_{G_0} |\nabla_{g(t_1)}u|_{g(t_1)}^2 \df x_{g(t_1)}=\frac{H(t_2)}{D(t_2)}-\frac{H(t_1)}{D(t_1)}.
    \end{equation}
	
	Since both $H$ and $D$ are Lipschitz with some constants depending only on $I$ and there exists $\delta>0$ such that $D(s)>\delta$ for all $s\in I$, we conclude that $\lambda_k(t)$ is Lipschitz in $I$. To compute $\dot\lambda_k(t_0)$, if it exists, we consider $u_0\in E(\lambda_k(t_0))$ such that $\int_{G_0} u_0^2\rho(t_0)\df x_{g(t_0)}=1$ and $u(t,x):=u_0(x)-P_t(u_0)(x)$. Let
	\begin{equation}
		\Phi(t):=\int_{G_0} |\nabla_{g(t)} u(t,x)|_{g(t)}^2-\lambda_k(t)u(t,x)^2\rho(t)\df x_{g(t)},\quad \forall t\in I.
	\end{equation}
	
	Observe that $\Phi(t)\ge 0$ for all $t\in I$ and $\Phi(t_0)=0$, so that $\dot{\Phi}(t_0)=0$, and the formula of $\dot\lambda_k(t_0)$ follows from the expansion of $\dot{\Phi}(t_0)$.
    \end{proof}

\subsection{Geometric properties of eigenspace with extremal pairs}\label{general_method}
Another ingredient for the proof of Theorem \ref{general_theorem} is the following lemma.

\begin{lemma}\label{lemma2_main2}
	Let $k\in\mbb{N}$ and suppose that $(g_*,\rho_*)\in\mathcal{A}_+(G)$ be an extremal pair for the functional $F(g,\rho)=\lambda_k(g,\rho)N(g,\rho)$, for some smooth function $N:\mathcal{A}_+(G)\to\mbb{R}_{>0}$ satisfying the homogeneity condition. Suppose further that there exists $(\phi,\eta)\in\mathcal{H}(g_*)=L^2(G,g_*)\times L^2(G,g_*)$ such that:
	\begin{equation}\label{inner_product=0}
		\langle \nabla N(g_*,\rho_*),(\phi,\eta)\rangle_{\mathcal{H}(g_*)}=0,
	\end{equation}
	then there exists $u\in E(\lambda_k(g_*,\rho_*))\backslash\{0\}$ such that
	\begin{equation}\langle Q(u;g_*,\rho_*),(\phi,\eta)\rangle_{\mathcal{H}(g_*)}=0.\end{equation}
\end{lemma}

\begin{proof}
	First, let us show that there exists some small $t$ such that $\dot \lambda_k(t)$ exists. Let $(\phi,\eta)\in \mathcal{H}(g_*)$ satisfy \eqref{inner_product=0}.
	Since $\mathcal{C}(G)$ (as given in \eqref{smooth_space}) is dense in $L^2(G,g_*)$, there exists $\phi_j,\eta_j\in C^2(G)$ such that $\phi_j\to\phi$ and $\eta_j\to \eta$ in the $L^2(G,g_*)$ topology and $(\phi_j,\eta_j)$ satisfies:
	\begin{equation}\label{inner_product=0_2}	\langle \nabla N(g_*,\rho_*),(\phi_j,\eta_j)\rangle_{\mathcal{H}(g_*)}=0,\quad \forall j.\end{equation}
	
	We define:
	\begin{equation}
    A_j(t):=\left(\frac{N(g_*,\rho_*)}{N(g_*+t\phi_j,\rho_*+t\eta_j)}\right)^{1/2},
    \end{equation}
	and 
    \begin{equation}
    g_j(t):=A_j(t)(g_*+t\phi_j),\quad  \rho_j(t):=A_j(t)(\rho_*+t\eta_j),\quad \forall j.
    \end{equation}
    
    Then, $A_j(0)=1,g_j(0)=g_*$, and $\rho_j(0)=\rho_*$. We write $F(t)=F(g_j(t),\rho_j(t))$ and, without loss of generality, suppose that $F(t)\le F(0)+o(t)$ as $t\to 0$. Observe that:
	\begin{equation}
	N(g_j(t),\rho_j(t))=N(g_*,\rho_*),
	\end{equation}
	so that ${\lambda_k}(t)\le {\lambda_k}(0)+o(t)$ as $t\to 0$. Moreover, by \eqref{inner_product=0_2}, we have that $\dot A_j(0)=0$, and consequently, $\dot g_j(0)=\phi_jg_*,\dot \rho_j(0)=\eta_j$. Let us fix $j$ and consider $\delta>0$ such that:
	\begin{equation}\delta\le\frac{\lambda_{k}(0)-\lambda_k(-\epsilon)}{\epsilon}=\frac{1}{\epsilon}\int_{-\epsilon}^0\dot\lambda_k(t)\df t\le \esssup_{t\in[-\epsilon,0]}\dot\lambda_k(t).\end{equation}
	
	Therefore, there exists $t\in(-\epsilon,0)$ such that $\dot \lambda_k(t)$ exists and $\dot\lambda_k(t)\ge 0$ for all sufficiently small $\epsilon>0$. Thus, we can pick a sequence $\{t_i^j\}\subset\mathbb{R}_{-}$ such that $\lim_{i\to \infty} t_i^j=0$ and $\dot \lambda_k(t_i^j)\ge 0$. Let $u_i^j\in E(\lambda_k(t_i^j))$ be such that $\int_{G_0} u^j_i\rho(t^j_i)\df x_{g(t^j_i)}=1$, then:
	 \begin{equation}
     \dot\lambda_k(t_i^j)=\langle Q(u^j_i;g(t^j_i),\rho(t^j_i)),(\phi_j,\eta_j)\rangle_{\mathcal{H}(g_*)}.
     \end{equation}
	 
	Since $u_i^j$ are eigenfunctions of $\Delta_{g(t_i^j)}$ on edges, $u_i^j$ are smooth on edges so that $u_i^j\in C^\infty(G)$ for all $i$. Then, the mean value theorem implies that $\{u_i^j\}$ is equicontinuous in $C^2(G)$. Thus, we can assume that $u_i^j$ converges to $u_-^{j}\in E(\lambda_k(0))$ as $i\to\infty$ in the $C^2(G)$ topology and $\int_{G_0} (u_-^j)^2\rho_*\df x_{g_*}=1$, since:
	\begin{equation}\Delta_{g_*} u_-^{j} =\lim_{i\to\infty}\Delta_{g(t_i^j)} u_i^j=\lim_{i\to\infty} \lambda_k(t_i^j) u_i^j\rho_j(t_i^j)=\lambda_k(0) u_-^{j}\rho(0).\end{equation}
	
	Observe that
	\begin{equation}\langle Q(u^j_-;g_*,\rho_*),(\phi_j,\eta_j)\rangle_{\mathcal{H}(g_*)}=\lim_{i\to\infty} \langle Q(u^j_i;g(t^j_i),\rho(t^j_i)),(\phi_j,\eta_j)\rangle_{\mathcal{H}(g_*)}=\lim_{i\to\infty}\dot\lambda_k(t_i^j)\ge 0.\end{equation}
	
	Similarly, there exists a sequence of eigenfunctions $u_+^{j}\in E(\lambda_k(0))$ such that
	\begin{equation}\int_{G_0} (u_+^j)^2\rho_*\df x_{g_*}=1, \quad\langle Q(u^j_+;g_*,\rho_*),(\phi_j,\eta_j)\rangle_{\mathcal{H}(g_*)}\le 0.\end{equation}
	
	By passing to subsequences, we can assume that $u_+^{j}\to u_+$ and $u_-^{j}\to u_-$ in the $C^2(G)$ topology with $u_{\pm}\in E(\lambda_k(0))$. Then, we have 
	\begin{equation}
    \langle Q(u_+;g_*,\rho_*),(\phi,\eta)\rangle_{\mathcal{H}(g_*)}=\lim_{j\to\infty}\langle Q(u_+^j;g_*,\rho_*),(\phi_j,\eta_j)\rangle_{\mathcal{H}(g_*)}\le 0,
    \end{equation}
	and
	\begin{equation} \langle Q(u_-;g_*,\rho_*),(\phi,\eta)\rangle_{\mathcal{H}(g_*)}=\lim_{j\to\infty} \langle Q(u_-^j;g_*,\rho_*),(\phi_j,\eta_j)\rangle_{\mathcal{H}(g_*)}\ge 0.\end{equation}
	To complete the proof, consider the family of eigenfunctions given by
	\begin{equation}U=\{su_++(1-s)u_-:s\in[0,1]\}.\end{equation}

    Then, either $u_+\equiv u_-$, in which case $\langle Q(u_+;g_*,\rho_*),(\phi,\eta)\rangle_{\mathcal{H}_0}=0$, or they are linearly independent, in which case there exists $u\in U$ such that $ \langle Q(u;g_*,\rho_*),(\phi,\eta)\rangle_{\mathcal{H}(g_*)}=0$.
\end{proof}

\begin{proof}[Proof of Theorem \ref{general_theorem}]
	Let $S$ be the convex hull of the set
	\begin{equation}\left\{-Q(u;g_*,\rho_*):u\in E(\lambda_k)\right\}
    \end{equation}
	in $\mathcal{H}(g_*)$, where $\lambda_k=\lambda_k(g_*,\rho_*)$; we show that $\nabla N(g_*,\rho_*)\in S$. If $\nabla N(g_*,\rho_*)\notin S$, the Hahn--Banach Theorem (second geometric version) implies the existence of $Y\in\mathcal{A}(G)$ such that:
	\begin{equation}
		\left\langle Y,\nabla N(g_*,\rho_*)\right\rangle_{\mathcal{H}(g_*)}>0>\left\langle Y,-Q(u;g_*,\rho_*)\right\rangle_{\mathcal{H}(g_*)},\quad \forall u\in E(\lambda_k)\backslash\{0\}.
	\end{equation}
	
	Let $X:=(g_*,\rho_*)$, then clearly $X\in\mathcal{H}(g_*)$ and $\langle Q(u;g_*,\rho_*),X\rangle_{\mathcal{H}(g_*)}<0$ for all $u\in E(\lambda_k)\backslash \{0\}$. Moreover, by considering the one-parameter family $(g(t),\rho(t))=(e^tg_*,e^t\rho_*)$, we have
    \begin{equation}
    \dot g(0)=g_*,\quad \dot\rho(0)=\rho_*,\quad F(g(t),\rho(t))=F(g_*,\rho_*).
    \end{equation}
    
    By differentiating $F(g(t),\rho(t))$ at $t=0$, we have $\langle \nabla N(g_*,\rho_*),X\rangle_{\mathcal{H}(g_*)}>0$. We consider:
	\begin{equation}
    \tilde{Y}:=Y-\frac{\langle Y,\nabla N(g_*,\rho_*)\rangle_{\mathcal{H}(g_*)}}{\langle X, \nabla N(g_*,\rho_*)\rangle_{\mathcal{H}(g_*)}}X,
    \end{equation}
	then clearly $\langle \tilde{Y}, \nabla N(g_*,\rho_*)\rangle_{\mathcal{H}(g_*)}=0$. However, for all $u\in E(\lambda_k)\backslash\{0\}$, we have:
	\begin{equation}
		\langle Q(u;g_*,\rho_*),\tilde{Y}\rangle_{\mathcal{H}(g_*)}
		=\langle Q(u;g_*,\rho_*), Y\rangle_{\mathcal{H}(g_*)}-\frac{\langle Y,\nabla N(g_*,\rho_*)\rangle_{\mathcal{H}(g_*)}}{\langle X,\nabla N(g_*,\rho_*)\rangle_{\mathcal{H}(g_*)}}\langle Q(u;g_*,\rho_*), X\rangle_{\mathcal{H}(g_*)}>0,
		\end{equation}
	which contradicts Lemma \ref{lemma2_main2}. Therefore, $\nabla N(g_*,\rho_*)\in S$ so that there exist some $\lambda_k$-eigenfunctions $f_1,f_2,\dots,f_n$ such that:
	\begin{equation}\label{general_proof_equation}
		\sum_{i=1}^n -Q(f_i;g_*,\rho_*)=\nabla N(g_*,\rho_*),
		\end{equation}
    and equation \eqref{general_main_equation} follows.
	
	Now, suppose that for a given pair $(g,\rho)\in\mathcal{A}_+(G)$, there exist some functions $f_1,\dots,f_n\in E(\lambda_k(g,\rho))$ satisfying \eqref{general_proof_equation}. We suppose that $\lambda_k(g,\rho)>\lambda_{k-1}(g,\rho)$ (the case $\lambda_k(g,\rho)<\lambda_{k+1}(g,\rho)$ can be proved similarly). Let $\mathcal{F}=\hbox{span}\{f_1,\dots,f_n\}$ and consider an arbitrary one-parameter smooth family of pairs $(g(t),\rho(t))$ with $g(0)=g$ and $\rho(0)=\rho$. We rescale $g(t)$ and $\rho(t)$ such that 
	\begin{equation}\label{normalised_LK}
	N(g(t),\rho(t))=N(g,\rho)=1,
	\end{equation}
	for all $t$. Let $\phi=\dot g(0)$ and $\eta=\dot \rho(0)$. Differentiating \eqref{normalised_LK} at $t=0$ gives us:
	\begin{equation}	0=\langle\nabla N(g,\rho),(\phi,\eta)\rangle_{\mathcal{H}(g_*)}=\sum_{i=1}^n -\langle Q(f_i;g_*,\rho_*),(\phi,\eta)\rangle_{\mathcal{H}(g_*)}.
    \end{equation}

	Therefore, there exist $f_\pm\in \mathcal{F}$ such that
	\begin{equation}
    \langle Q(f_+;g_*,\rho_*),(\phi,\eta)\rangle_{\mathcal{H}(g_*)}\le 0,\quad \langle Q(f_-;g_*,\rho_*),(\phi,\eta)\rangle_{\mathcal{H}(g_*)}\ge 0.
    \end{equation}
	
	We rescale $f_\pm$ such that $\int_{G_0} f_\pm^2\rho \df x_g=1$. By Lemma \ref{derivative_lambda}, instead of computing the derivative of $\lambda_k$ at $t=0$ directly, we compute left and right derivatives to get:
	\begin{equation}
    \lim_{t\to 0^+} \frac{\lambda_k(t)-\lambda_k(0)}{t}=\langle Q(f_+;g_*,\rho_*),(\phi,\eta)\rangle_{\mathcal{H}(g_*)} \le 0,
    \end{equation}
    and
	\begin{equation}
	    \lim_{t\to 0^-} \frac{\lambda_k(t)-\lambda_k(0)}{t}=\langle Q(f_-;g_*,\rho_*),(\phi,\eta)\rangle_{\mathcal{H}(g_*)}\ge 0.
	\end{equation} 
	Hence, $\lambda_k(t)\le\lambda_k+o(t)$ as $t\to 0$. Since the smooth family of pairs $(g(t),\rho(t))$ is chosen arbitrarily, $(g,\rho)$ must be an extremal pair.
\end{proof}
\section{Proof of the main theorems and their consequences}
\subsection{Extremal metrics for normalised eigenvalues}
We will skip the proof of Theorem \ref{main_theorem1}, since the proof follows directly from Theorem \ref{main_theorem2} by replacing the space $\mathcal{A}=\mbb{R}^{|E|}_{+}\times \mathcal{C}(G)$ by $\tilde{\mathcal{A}}=\mbb{R}^{|E|}_{+}\times\{1\}$, i.e. fixing $\rho\equiv 1$. Recall from Theorem \ref{main_theorem1} that for a finite compact connected graph $G$ and an an extremal metric $g$ of the functional $\ol{\lambda_k}(\cdot)$ for some $k\in\mbb{N}$, there exists $\lambda_k(g)$-eigenfunctions $f_1,f_2,\dots,f_n$ such that:
\begin{equation}
\sum_{i=1}^n \left[|\nabla_{g}f_i|_g^2+\lambda_kf_i^2\right]=1.
\end{equation}
Let $F=(f_1,f_2,\dots,f_n)$, then
\begin{equation}\label{eqn_lambda_4lambda}
	\Delta_g\left(|F|^2-\frac{1}{2\lambda_k}\right)=2\sum_{i=1}^n\left[ f_i\Delta_g f_i- |\nabla_{g}f_i|_g^2\right]=4\lambda_k\left(|F|^2-\frac{1}{2\lambda_k}\right).
\end{equation}

Thus, if $4\lambda_k$ is not an eigenvalue of $\Delta_g$, then $|F|\equiv(2\lambda_k)^{-1/2}$ so that by a suitable scaling factor, we have an isometric minimal immersion from $(G,g)$ to the unit sphere. Conversely, if $|F|\not\equiv (2\lambda_k)^{-1/2}$, then $4\lambda_k$ is also an eigenvalue and $|F|^2-1/(2\lambda_k)$ is an eigenfunction. In particular, we have:

\begin{proposition}\label{4lambda_theorem}
	Let $(G,g)$ be a quantum graph. Suppose that $\lambda$ is an eigenvalue and $f_1,f_2,\dots,f_n$ are $\lambda$-eigenfunctions such that
	\begin{equation}
    \sum_{i=1}^n \left(|\nabla_gf_i|_g^2+\lambda f_i^2\right)=1.
    \end{equation}
	Let $F:=(f_1,f_2,\dots,f_n)$, then:
	\begin{enumerate}
		\item If $4\lambda$ is not an eigenvalue, then $|F|=(2\lambda)^{-1/2}$. Moreover, let $\tilde{F}:=\sqrt{2} F$ and $R=\lambda^{-1/2}$, then $\tilde{F}: (G,g)\to\mbb{S}^{n-1}_{R}$ is an isometric minimal immersion.
		\item If $|F|\not\equiv (2\lambda)^{-1/2}$, then $4\lambda$ is an eigenvalue with an eigenfunction $|F|^2-(2\lambda)^{-1}$.
	\end{enumerate}
\end{proposition}

\subsection{Extremal pairs for naturally normalised eigenvalues} We first prove Theorem \ref{main_theorem2}. Recall the normalisation of $\lambda_k(g,\rho)$:
\begin{equation}
\ol{\lambda_k}(g,\rho)=\lambda_k(g,\rho)N(g,\rho),\quad N(g,\rho)=L(g)\int_{G_0} \rho \df x_g.
\end{equation}
Then,
\begin{equation}
    \nabla N(g,\rho)=\left(\frac{1}{2g}\left(\int_{G_0} \rho \df x_g+\rho L(g)\right),L(g)\right).
\end{equation}

By Theorem \ref{general_theorem}, there exist some functions $f_1,f_2,\dots,f_n\in E(\lambda_k(g_*,\rho_*))$ such that:
\begin{equation}
\sum_{i=1}^n |\nabla_{g_*} f_i|_{g_*}^2=\int_{G_0} \rho_* \df x_{g_*},\quad\text{and}\quad \sum_{i=1}^n f_i^2=\frac{L(g_*)}{\lambda_k(g_*,\rho_*)}.
\end{equation}

From the second identity, observe that
\begin{equation}
    0=\frac{1}{2}\sum_{i=1}^n \Delta_{g_*} (f_i^2)=\sum_{j=1}^n \left(f_i \Delta_{g_*} f_i -|\nabla_{g_*} f_i|_{g_*}^2\right)=\rho_*L(g_*)-\int_{G_0} \rho_*\df x_{g_*},
\end{equation}
so that $\rho_*$ is a constant function. To complete the proof, we rescale $f_i$ by a suitable factor.

From Proposition \ref{eqn_lambda_4lambda} and Theorem \ref{main_theorem2}, we have the following corollaries.

\begin{corollary}
	Let $G$ be a finite compact connected graph and suppose that $g_*$ is extremal for the functional $\ol{\lambda_k}(g)$. Suppose further that $4\lambda_k(g_*)$ is not an eigenvalue and either $\lambda_{k}(g_*)>\lambda_{k-1}(g_*)$ or $\lambda_k(g_*)<\lambda_{k+1}(g_*)$. Then $(g_*,1)$ is extremal for the functional $\ol{\lambda_k}(g,\rho)$.
	\end{corollary}

\begin{corollary}\label{maximiser_minimiser_if_exists}
	Let $G$ be a finite compact connected graph. Suppose that $(g_*,\rho_*)$ is extremal for the functional $\ol{\lambda_k}(g,\rho)$. Then $\rho_*$ is a constant and $g_*$ is extremal for the functional $\ol{\lambda_k}(g)$.
\end{corollary}

\begin{proof}
    We prove the result when
    \begin{equation}
        \ol{\lambda_k}(g(t),\rho(t))\le \ol{\lambda_k}(g_*,\rho_*)+o(t), \quad\text{as}\quad t\to 0,
    \end{equation}
    for all one-parameter smooth family of pairs $\{(g(t),\rho(t)\}\subset\mathcal{A}_+$ such that $g(0)=g_*$ and $\rho(0)=\rho_*$, and the latter case can be proven similarly. Then, for any arbitrary one-parameter smooth families of metric $\{g(t)\}$ such that $g(0)=g_*$, we consider the following one-parameter smooth family of pairs $\{(g(t),\rho_*)\}\subset \mathcal{A}_+$. Since $\rho_*$ is a constant function, we have
    \begin{equation}
        \ol{\lambda_k}(g(t))=\ol{\lambda_k}(g(t),\rho_*)\le \ol{\lambda_k}(g_*,\rho_*)+o(t)=\ol{\lambda_k}(g_*)+o(t),
    \end{equation}
    as $t\to 0$. Thus, $g_*$ is extremal for the functional $\ol{\lambda_k}(\cdot)$.	
	\end{proof}

\subsection{Extremal pairs for general normalised eigenvalues}
We now prove Theorem \ref{main_theorem3}. Recall the formula of $\ol{\lambda_k}^{(\alpha)}(g,\rho)$ for $\alpha\in\mbb{R}$, 
\begin{equation}F_\alpha(g)=\int_{G_0} g^{1-\alpha} \df x,\quad H_\alpha(g,\rho)=\int_{G_0} \rho g^\alpha \df x, \quad N_\alpha(g,\rho)=F_\alpha(g)H_\alpha(g,\rho),
\end{equation}
and $\ol{\lambda_k}^{(\alpha)}(g,\rho)=\lambda_k(g,\rho)N_\alpha(g,\rho)$. Then, a simple calculation shows that:
	\begin{equation}
    \nabla N_\alpha (g,\rho)=\left(\alpha\rho g^{\alpha-3/2}F_\alpha(g)+(1-\alpha) g^{-\alpha+1/2}H_\alpha(g,\rho),g^{\alpha-1/2}F_\alpha(g)\right),
    \end{equation}
and by Theorem \ref{general_theorem}, there exist $f_1,\dots,f_n\in E(\lambda_k(g_*,\rho_*))$ such that:
\begin{equation}\label{general_1}
    \sum_{i=1}^n |\nabla_{g_*} f_i|_{g_*}^2=2(1-\alpha)g_*^{-\alpha+1/2}H_\alpha(g_*,\rho_*)+(2\alpha-1)g_*^{\alpha-1/2}\rho_*F_\alpha(g_*),
\end{equation}
and
\begin{equation}\label{general_2}
    \lambda_k(g_*,\rho_*)\sum_{i=1}^n f_i^2=g_*^{\alpha-1/2}F_\alpha(g_*).
\end{equation}
	Since $\alpha\ne 1/2$, \eqref{general_2} implies that $g_*$ is a regular metric, by the continuity of $W^{1,2}$ functions on $G$. Observe that:
		\begin{equation}
        \begin{aligned}
			0&=\frac{1}{2}\sum_{i=1}^n \Delta_{g_*}(f_i)^2=\sum_{i=1}^n\left[ f_i\Delta_{g_*} f_i -|\nabla_{g_*} f_j|_{g_*}^2\right]=2(1-\alpha)\left[g_*^{\alpha-1/2} \rho_* F_\alpha(g_*)-g_*^{-\alpha+1/2}H_{\alpha}(g_*,\rho_*)\right].
            \end{aligned}
		\end{equation}
        
		Since $\alpha\ne 1$, we have:
        \begin{equation}
        \rho_*=g_*^{-2\alpha+1}\cdot\frac{H_\alpha(g_*,\rho_*)}{F_{\alpha}(g_*)},
        \end{equation}
		so that $\rho_*$ is a constant function. To finish the proof, we rescale functions $f_i$ with a suitable factor.

\section{Normalised eigenvalue bounds}\label{eigenvalue_bounds}
In this section, we obtain upper and lower bounds for the naturally normalised smallest positive eigenvalue, $\ol{\lambda_1}(\cdot,\cdot)$. We note that for the eigenvalue functional $\ol{\lambda_1}(g)$, universal upper and lower bounds were already obtained: In \cite{lower_bound_1,lower_bound_2,lower_bound_3}, we have 
\begin{equation}\label{lower_simple_eigenvalue}
    \ol{\lambda_1}(g)\ge \pi^2,\quad \forall g\in\mbb{R}^{|E|}_+,
\end{equation} and in \cite{Berkolaiko_2017}, there exists a constant $C=C(G)$ such that 
\begin{equation}
\ol{\lambda_1}(g)\le \pi^2 C,\quad \forall g\in\mbb{R}^{|E|}_+.    
\end{equation}

We prove that for every finite compact connected graph $G$, one has
\begin{equation}\label{eigenvalue_blows_up}
    \sup_{(g,\rho)\in \mathcal{A}_+(G)}\ol{\lambda_1}(g,\rho)=\infty,
\end{equation}
and
\begin{equation}\label{lower_bound_eigenvalue_1}
    \inf_{(g,\rho)\in {\mathcal{A}_+}(G)}\ol{\lambda_1}(g,\rho)\ge1.
\end{equation}

\subsection{Supremum problem on graphs with at least two edges}\label{proof_large_eigenvalue_1}
We consider the following example to show that the normalised eigenvalue can be arbitrarily large.
\begin{example}\label{large_eigenvalue}
    Let $G$ be a flower graph with $m$ edges, i.e. a graph with one vertex and $m$ loops (see Figure \ref{flower_picture}). Consider a pair $(g,\rho)\in\mathcal{A}_+(G)$ and denote the restriction of $\rho$ to an edge $e_j$ as $\rho_j$ and the length of the edge $e_j$ as $\ell_j$, for all $j$. Suppose further that $\rho_j$ is a constant function for all $j$. Then, a simple computation shows that $\lambda_1(g,\rho)$ is the smallest positive solution of
    \begin{equation}
        \sum_{j=1}^m\rho_j\tan\left(\frac{\ell_j\sqrt{\rho_j\lambda}}{2}\right)=0.
    \end{equation}

      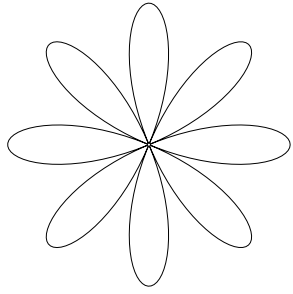
\begin{figure}[h]
 	\centering
  \begin{tikzpicture}[scale=0.6]
 	\begin{polaraxis}[grid=none, axis lines=none]
 		\addplot[mark=none,domain=0:360,samples=300] {cos(x*4)};
 		\draw[fill = white!50] (0,0) circle (1pt);
 	\end{polaraxis}
 \end{tikzpicture}
 
 	\caption{A flower graph with $8$ edges.}
 	\label{flower_picture}
 \end{figure}

    Suppose that $\ell_1\sqrt{\rho_1}=\max_{j}(\ell_j\sqrt{\rho_j})$, then we have
    \begin{equation}
        \lambda_1(g,\rho)\ge \frac{\pi^2}{\ell_1^2\rho_1},
    \end{equation}
    so that
    \begin{equation}\label{estimate_star}
    \begin{aligned}
        \ol{\lambda_1}(g,\rho)&\ge\frac{\pi^2}{\ell_1^2\rho_1}\left(\sum_{j=1}^m\ell_j\right)\left(\sum_{j=1}^m \ell_j\rho_j\right)=\pi^2\left(\sum_{j=1}^m \frac{\ell_j}{\ell_1}\right)\left(\sum_{j=1}^m\frac{\ell_j}{\ell_1}\cdot\frac{\rho_j}{\rho_1}\right)
        \end{aligned}
    \end{equation}

    We choose
    \begin{equation}
        \ell_j=\begin{cases}
            2^{-n},&\quad\text{if}\;j=1;\\
            1,&\quad\text{otherwise},
        \end{cases}
        \quad\text{and} \quad \rho_j=\begin{cases}
            2^{2n+1},&\quad \text{if}\; j=1;\\
            1,&\quad \text{otherwise},
        \end{cases}
    \end{equation}
    where $n\in\mbb{N}$. Then, inequality \eqref{estimate_star} implies that
    \begin{equation}
         \ol{\lambda_1}(g,\rho)\ge {\pi^2}\left(1+(m-1)2^{n}\right)\left(1+(m-1)2^{-n-1}\right)\to\infty,
    \end{equation}
    as $n\to\infty$.
    \end{example}

   We use the pair of metric and density functions constructed in Example \ref{large_eigenvalue} to construct appropriate pairs on other discrete graphs. We begin by recalling a beautiful technique from \cite{lower_bound_3}, where we construct a quantum loop from a quantum graph such that the eigenvalue from the loop is smaller than the one from the original quantum graph. The construction is purely combinatorial, and we begin by recalling the definition of a double cover of a quantum graph.

    \begin{definition}
        Let $(G,g)$ be a finite compact quantum graph. We construct a quantum graph $(\tilde{G},\tilde{g})$ by doubling each edge of $(G,g)$ such that if $\tilde{e}$ is the double edge of $e\in E(G)$, then $\ell_{\tilde{e}}=\ell_e$. We say that $(\tilde{G},\tilde{g})$ is the double cover of $(G,g)$.
    \end{definition}

    We note that for each function $f:G_0\to\mbb{R}$, there exists a natural symmetric extension $\tilde{f}:\tilde{G}_0\to\mbb{R}$ of $f$ such that if $\tilde{e}$ is the doubling edge of $e\in E(G)$, then we have $\tilde{f}_{\tilde{e}}\equiv f_e$. Therefore, there is a natural embedding from $W^{1,2}(G)$ to $W^{1,2}(\tilde{G})$, which maps $f\in W^{1,2}(G)$ to $\tilde{f}\in W^{1,2}(\tilde{G})$. Moreover, for every pair $(g,\rho)\in \mathcal{A}_+(G)$, we have 
    \begin{equation}
        R(G,g,\rho;f)=R(\tilde{G},\tilde{g},\tilde{\rho};\tilde{f}),\quad \forall f\in W^{1,2}(G),
    \end{equation}
    so that
    \begin{equation}
        \lambda_k(G,g,\rho)\ge \lambda_k(\tilde{G},\tilde{g},\tilde{\rho}),\quad \forall k\in\mbb{N}.
    \end{equation}

    We note that every vertex in $\tilde{G}$ has even degree. By a classical result in discrete graph theory, there exists an Eulerian path $\mathcal{P}(\tilde{G})$ on $\tilde{G}$, and we consider $\mathcal{P}(\tilde{G})$ as a loop. If $h$ is a metric, not necessarily a symmetric metric, on $\tilde{G}$, then there is a metric on $\mathcal{P}(\tilde{G})$, also called $h$, such that $L(\mathcal{P}(\tilde{G}),h)=L(\tilde{G},h)$. Moreover, for each function $u\in W^{1,2}(\tilde{G})$, we can consider $u$ as a function on $W^{1,2}(\mathcal{P}(\tilde{G}))$, so that we have
    \begin{equation}\label{compare_graph_loop}
        \lambda_k(G,g,\rho)\ge \lambda_k(\tilde{G},\tilde{g},\tilde{\rho})\ge \lambda_k(\mathcal{P}(\tilde{G}),\tilde{g},\tilde{\rho}),\quad \forall k\in\mbb{N}.
    \end{equation}

\begin{lemma}\label{large_eigenvalue_tree}
    Let $G$ be a finite compact connected tree with $|E|\ge 2$. Then, 
    \begin{equation}
    \sup_{(g,\rho)\in\mathcal{A}_+(G)}\ol{\lambda_1}(g,\rho)=\infty.
 \end{equation}
\end{lemma}

\begin{proof}
    Let $v$ be a vertex in $G$ with degree at least two. We decompose $G$ into subgraphs $G_1,\dots,G_m$ such that $\bigcup G_i=G$ and $G_i\cap G_j=\{v\}$ for all distinct $i,j$. Then, each $G_i$ is a tree. Suppose that $g$ is a metric on $G$ and we consider a density function $\rho$ on $G$ such that $\rho_j=\rho|_{G_j}$ is a constant function. We consider a double cover $\tilde{G_i}$ and an Eulerian path $\mathcal{P}(\tilde{G}_i)$ on each $G_i$. We consider $\mathcal{P}(\tilde{G}_i)$ as a loop whose vertex is $v$. Let $H$ be the graph formed by attaching $\mathcal{P}(\tilde{G}_i)$ to $v$ for all $i$. Then, $H$ is a flower graph with $m$ edges. We consider a metric $\tilde{g}$ on $H$ such that $\ell_{\mathcal{P}(\tilde{G_i})}=2L(G_i,g|_{G_i})$ for all $i$ and consider a density function $\tilde{\rho}$ on $H$ such that $\tilde{\rho}|_{\mathcal{P}(\tilde{G_i})}=\rho|_{G_i}$. Since
    \begin{equation}
        L(H,\tilde{g})=2L(G,g),\quad \text{and}\quad \int_{H_0} \tilde\rho \df x_{\tilde g}=2\int_{G_0}\rho\df x_g,
    \end{equation}
    we have
    \begin{equation}
        \ol{\lambda_1}(G,g,\rho)\ge \frac{1}{4}\ol{\lambda_1}(H,\tilde{g},\tilde{\rho}).
    \end{equation}

    Now we consider a pair $(g,\rho)$ given as follows
    \begin{equation}
        L(G_j,g|_{G_j})=\begin{cases}
            2^{-n},&\quad\text{if}\;j=1;\\
            1,&\quad\text{otherwise},
        \end{cases}
        \quad\text{and} \quad \rho|_{G_j}=\begin{cases}
            2^{2n+1},&\quad \text{if}\; j=1;\\
            1,&\quad \text{otherwise}.
        \end{cases}
    \end{equation}

    Then, Example \ref{large_eigenvalue} shows that
    \begin{equation}
        \ol{\lambda_1}(G,g,\rho)\ge \frac{1}{4}\ol{\lambda_1}(H,\tilde{g},\tilde{\rho})\ge \frac{\pi^2}{4}\left(1+(m-1)2^{n}\right)\left(1+(m-1)2^{-n-1}\right)\to\infty,
    \end{equation}
    as $n\to\infty$.
\end{proof}

\begin{remark}\label{why_cannot_apply_interval}
    We note that if $G$ is a simple interval graph, then $H$ is a loop. Therefore, the density function given in Example \ref{large_eigenvalue} is a constant function, so that the normalised smallest positive eigenvalue of $H$ is always $4\pi^2$ for all pairs of metric and constant density function. Thus, the proof of Lemma \ref{large_eigenvalue_tree} cannot apply to the case when $G$ is an interval.
\end{remark}

To complete the proof of \eqref{eigenvalue_blows_up} for the case $|E(G)|\ge 2$, we need another ingredient: the monotonicity of eigenvalues through the glueing action. The following proposition is a special case of \cite[Theorem 3.1.10]{intro_quantum_graph}, where Berkolaiko and Kuchment obtained a monotonicity result for eigenvalues of Schr\"odinger operators with $\delta$-coupling conditions.

\begin{proposition}\label{gluing_action}
    Let $G$ be a discrete graph. We form a graph $\hat{G}$ from $G$ by glueing two vertices $v_1,v_2\in V(G)$ to a single vertex $v_0$. Then, for every pair $(g,\rho)\in \mathcal{A}_+(G)$, we can consider them as a pair in $\mathcal{A}_+(\hat G)$ and one has
    \begin{equation}
        \lambda_k(G,g,\rho)\le \lambda_k(\hat{G},g,\rho),\quad \forall k\in\mbb{N}.
    \end{equation}
\end{proposition}

\begin{proof}
The proof follows from the fact that there is an embedding from $W^{1,2}(\hat G)$ to $W^{1,2}(G)$, and note that
\begin{equation}
    R(G,g,\rho;f)=R(\hat G,g,\rho;f),\quad \forall f\in W^{1,2}(\hat G).
\end{equation}
\end{proof}

\begin{theorem}
 Let $G$ be a finite compact connected graph with at least two edges. Then \begin{equation}
 \sup_{(g,\rho)\in\mathcal{A}_+(G)}\ol{\lambda_1}(g,\rho)=\infty.
 \end{equation}
\end{theorem}
\begin{proof}
    We pick a vertex $v\in V(G)$ with $\deg v\ge 2$ and let $e$ be an edge connecting to $e$ such that if we disconnect $e$ at $v$, then the connectivity of $G$ remains the same. We form a new graph by disconnecting $e$ at $v$, so that this new graph has a smaller Betti number than the one of $G$. We repeat this process until the Betti number of the graph is reduced to zero, i.e. the graph formed is a tree. We call this tree $H$ and observe that every pair $(g,\rho)\in \mathcal{A}_+(G)$ can be considered as a pair $(g,\rho)\in\mathcal{A}_+(H)$ and vice versa. By Proposition \ref{gluing_action} and Lemma \ref{large_eigenvalue_tree}, there exists a sequence of pairs $\{(g^{(n)},\rho^{(n)})\}\subset\mathcal{A}_+(G)$ such that
    \begin{equation}
        \ol{\lambda_1}(G,g^{(n)},\rho^{(n)})\ge\ol{\lambda_1}(H,g^{(n)},\rho^{(n)})\to\infty,
    \end{equation}
    as $n\to \infty$.
\end{proof}

\subsection{Supremum problem on interval graphs}\label{proof_large_eigenvalue_2}
It is left to prove \eqref{eigenvalue_blows_up} when $G$ is a loop or an interval. By Proposition \ref{gluing_action}, we only need to prove when $G$ is an interval.

We begin the proof by considering the following example:
\begin{example}\label{large_eigenvalue_interval_example}
    Let $J$ be a joint of two interval graphs $J_1,J_2$ as given in Figure \ref{interval_large}. We show that there exists a sequence of pair $(g^{(n)},\rho^{(n)})\subset\mathcal{A}_+(J)$ such that $\ol{\lambda_1}(g^{(n)},\rho^{(n)})\to\infty$. 

        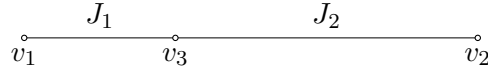
\begin{figure}[h]
        \centering
        \begin{tikzpicture}
    \tkzSetUpPoint[fill=white]
    \tkzDefPoint(0,0){A}
    \tkzDefPoint(2,0){B}
    \tkzDefPoint(1,0){E}
    \tkzDefPoint(2,0){B}
    \tkzDefPoint(0,0){A}
    \tkzDefPoint(4,0){F}
    \tkzDefPoint(6,0){C}
    \tkzDefPoint(7,0){G}
    \tkzDefPoint(2,0){B}
    \tkzDefPoint(8,0){D}
    \tkzDrawSegments(A,C)
    \tkzLabelPoint(A){$v_1$}
    \tkzLabelPoint(B){$v_3$}
    \tkzLabelPoint(C){$v_2$}
    \tkzLabelPoint[above](E){$J_1$}
    \tkzLabelPoint[above](F){$J_2$}
    \tkzDrawPoints(A,B,C)
\end{tikzpicture}
        \caption{Construction of $J$}
        \label{interval_large}
    \end{figure}

    We consider a metric $g$ on $J$ with length $\ell_1$ on $J_1$ and $\ell_2$ on $J_2$, and a piecewise smooth density $\rho$ on $J$ such that $\rho$ is a constant $\rho_1$ on $J_1$ and $\rho$ is a constant $\rho_2$ on $J_2$. Then, a simple computation shows that $\lambda_1(g,\rho)$ is the smallest positive solution of the following equation
\begin{equation}\label{eigenvalue_equation_1}
    \sqrt{\rho_1}\tan\left(\ell_1\sqrt{\lambda\rho_1}\right)+\sqrt{\rho_2}\tan\left(\ell_2\sqrt{\lambda\rho_2}\right)=0.
\end{equation}

Without loss of generality, suppose that $\ell_1\sqrt{\rho_1}\ge \ell_2\sqrt{\rho_2}$, so that
\begin{equation}
    \lambda_1(g,\rho)\ge \frac{\pi^2}{4\ell_1^2\rho_1},
\end{equation}
and we have
\begin{equation}
\begin{aligned}
    \ol{\lambda_1}(g,\rho)&> \frac{\pi^2}{4\ell_1^2\rho_1}(\ell_1+\ell_2)(\ell_1\rho_1+\ell_2\rho_2)\ge\frac{\pi^2}{4} \cdot \frac{\ell_2}{\ell_1}
    \end{aligned}
\end{equation}

We choose $g^{(n)}$ and $\rho^{(n)}$ such that $\ell^{(n)}_1=2^{-n},\ell_2^{(n)}=1,\rho_1=2^{2n+1}$ and $\rho_2=1$, then,
\begin{equation}
    \ol{\lambda_1}(g^{(n)},\rho^{(n)})\ge \frac{\pi^2}{4}\cdot 2^n\to\infty,
    \end{equation}
    as $n\to\infty$.
\end{example}

In Example \ref{large_eigenvalue_interval_example}, if we consider $J$ as a single interval, then $g^{(n)}$ is a metric of length $\ell=\ell_1+\ell_2$, and the density function $\rho^{(n)}$ is piecewise smooth. However, these density functions are not in the class $\mathcal{C}_+$. To resolve this, we approximate piecewise smooth functions by smooth functions and obtain continuity of eigenvalues in density functions. We begin by obtaining the monotonicity of eigenvalues in density functions.

\begin{proposition}\label{monotonicity_of_density}
    Let $G$ be a finite compact graph. Suppose that $\rho,\tilde{\rho}\in\mathcal{C}_+(G)$ be two density functions such that $\rho\le \tilde{\rho}$ pointwise in $G_0$. Then,
    \begin{equation}
        \lambda_k(g,\rho)\ge \lambda_k(g,\tilde{\rho}),
    \end{equation}
    for all metric $g$ on $G$ and for all $k\in\mbb{N}$.
\end{proposition}

\begin{proof}
    Let $F\subset W^{1,2}(G)$ be a subspace such that $\dim F=k+1$ and 
    \begin{equation}
        \lambda_k(g,\rho)\ge\max_{f\in F\backslash\{0\}}\frac{\int_{G_0}|\nabla_g f|^2\df x_g}{\int_{G_0} f^2\rho\df x_g}.
    \end{equation}

    Then, 
    \begin{equation}
        \lambda_k(g,\tilde{\rho})\le \max_{f\in F\backslash\{0\}}\frac{\int_{G_0}|\nabla_g f|^2\df x_g}{\int_{G_0} f^2\tilde{\rho}\df x_g}\le \max_{f\in F\backslash\{0\}}\frac{\int_{G_0}|\nabla_g f|^2\df x_g}{\int_{G_0} f^2\rho\df x_g}\le \lambda_k(g,\rho).
    \end{equation}
\end{proof}

\begin{proposition}\label{estimate_difference_eigenvalue}
    Let $(G,g)$ be a connected quantum graph. Then, there exists a constant $C=C(g)$ such that for any piecewise positive smooth functions $\rho,\tilde\rho$, one has
    \begin{equation}
        |\lambda_k(g,\rho)-\lambda_k(g,\tilde{\rho})|\le C\Lambda_k\left(\Lambda_k+\frac{1}{m}\right)\frac{M}{m}\|\rho-\tilde\rho\|_{L^1(G,g)}\quad\forall k\in\mbb{N},
    \end{equation}
    where
    \begin{equation}
        m:=\min\{\rho(x),\tilde{\rho}(x):x\in (G,g)\},\quad M:=\max\{\rho(x),\tilde{\rho}(x):x\in (G,g)\}
    \end{equation}
    and
    \begin{equation}
        \Lambda_k:=\max\{\lambda_k(g,\rho),\lambda_k(g,\tilde{\rho})\}.
    \end{equation}
\end{proposition}

To begin with the proof, we note that there is an embedding from $W^{1,2}(G)$ to $L^\infty(G)$, so that there exists a constant $C=C(g)>0$ such that:
\begin{equation}
    \|f\|_{L^\infty(G,g)}^2\le C\|f\|_{W^{1,2}(G,g)}^2 ,\quad \forall f\in W^{1,2}(G,g).
\end{equation}

\begin{proof}
    Without loss of generality, suppose that $\lambda_k(g,\rho)\le \lambda_k(g,\tilde\rho)$. Let $F\subset W^{1,2}(G)$ be a $(k+1)$-dimensional subspace such that
    \begin{equation}
        \lambda_k(g,\rho)\ge \max_{f\in F\backslash\{0\}}\frac{\int_{G_0}|\nabla_g{f}|_g^2\df x_g}{\int_{G_0} f^2\rho\df x_{g_0}}
    \end{equation}

    Then, by the variational characterisation, we have
    \begin{equation}\label{estimate_difference_eigenvalue_1}
        \begin{aligned}
            \lambda_k(g,\tilde\rho)-\lambda_{k}(g,\rho)&\le \max_{f\in F\backslash\{0\}}\left(\frac{\int_{G_0}|\nabla_g{f}|_g^2\df x_g}{\int_{G_0} f^2\tilde\rho\df x_{g_0}}-\frac{\int_{G_0}|\nabla_g{f}|_g^2\df x_g}{\int_{G_0} f^2\rho\df x_{g_0}}\right)\\
            &\le\max_{f\in F\backslash\{0\}}\left(\frac{\int_{G_0}|\nabla_g{f}|_g^2\df x_g}{\int_{G_0} f^2\rho\df x_{g_0}}\cdot\frac{\int_{G_0}f^2(\rho-\tilde{\rho})\df x_g}{ \int_{G_0} f^2\tilde{\rho}\df x_g}\right)\\
            &\le \lambda_k(g,\rho)\|\rho-\tilde{\rho}\|_{L^1(G,g)} \max_{f\in F\backslash\{0\}}\frac{\|f\|_{L^\infty(G,g)}^2}{m\|f\|_{L^2(G,g)}^2}
        \end{aligned}
    \end{equation}

    Notice that for all functions $f\in F$ such that $\int_{G_0} f^2\rho \df x_g=1$, we have
    \begin{equation}\label{estimate_difference_eigenvalue_2}
        \frac{1}{M}\le\|f\|_{L^2(G,g)}^2\le \frac{1}{m},
    \end{equation}
    and
    \begin{equation}\label{estimate_difference_eigenvalue_3}
    \begin{aligned}
         \|f\|_{L^\infty(G,g)}^2&\le C\|f\|_{W^{1,2}(G,g)}^2\le C\left(\lambda_k(g,\rho)+\|f\|^2_{L^2(G,g)}\right)\le C\left(\Lambda_k+\frac{1}{m}\right).
    \end{aligned}
    \end{equation}

    Our assertions follow by applying inequalities \eqref{estimate_difference_eigenvalue_2} and \eqref{estimate_difference_eigenvalue_3} to \eqref{estimate_difference_eigenvalue_1}.
\end{proof}

\begin{lemma}\label{continuity_eigenvalue_density}
    Let $(G,g)$ be a finite compact connected quantum graph. Suppose that $\rho,\rho^{(n)}$ are piecewise smooth in $(G,g)$ such that $\rho^{(n)}\to\rho$ as $n\to\infty$ in $L^1(G)$ topology. Suppose further that there exist some constants $m,M$ such that
    \begin{equation}
        0<m<\inf\{\rho(x),\rho^{(n)}(x):x\in G_0, n\in\mbb{N}\}\le \sup\{\rho(x),\rho^{(n)}(x):x\in G_0, n\in\mbb{N}\}<M<\infty. 
    \end{equation}
    
    Then, for every $k\in\mbb{N}$, one has $\lambda_k\left(g,\rho^{(n)}\right)\to\lambda_k(g,\rho)$ as $n\to\infty$.
\end{lemma}

\begin{proof}
    By Proposition \ref{monotonicity_of_density}, we have
    \begin{equation}
       \max\{\lambda_k(g,\rho),\lambda_k\left(g,{\rho}^{(n)}\right)\}\le\lambda_k(g,m)<\infty,\quad \forall n,
    \end{equation}
    and our assertions follow directly from Proposition \ref{estimate_difference_eigenvalue} (note that the space $L^1(G)$ is equivalent to the space $L^1(G,g)$, so that we can replace the norm $\|\cdot\|_{L^1(G,g)}$ by $C\|\cdot\|_{L^1(G)}$ for some constant $C=C(g)>0$).
\end{proof}

We now prove \eqref{eigenvalue_blows_up} for interval graphs.

\begin{theorem}\label{pair_large_eigenvalue}
    Let $I$ be an interval graph. Then, there exists a sequence of smooth density functions $\{\rho^{(n)}\}$ and a sequence of metric $\{g^{(n)}\}$ on $I$ such that
    \begin{equation}
        \ol{\lambda_1}\left(g^{(n)},\rho^{(n)}\right)\to\infty,
    \end{equation}
    as $n\to\infty$.
\end{theorem}
\begin{proof}
    Let the vertices of $I$ be $\{v_1,v_2\}$ and we consider a new graph $J$ by imposing a vertex $v_3$ between $v_1$ and $v_2$. Let $J_1$ and $J_2$ be the edges of $v_1v_3$ and $v_2v_3$, respectively. Then, $J$ can be considered as the joint of two interval graphs. By Example \ref{large_eigenvalue_interval_example}, there exists a sequence of pair $(g^{(n)},\rho^{(n)})\subset \mathcal{A}_+(J)$ such that $\ol{\lambda_1}(J,g^{(n)},\rho^{(n)})\to\infty$. 

    We consider $\rho^{(n)}$ as a piecewise smooth function on the interval $[0,\ell^{(n)}]$, where $\ell^{(n)}=L(J,g^{(n)})$. By Lemma \ref{continuity_eigenvalue_density} and density of smooth functions in the $L^1(0,\ell^{(n)})$ topology, there exists a sequence of density function $\hat{\rho}^{(n)}\in C^\infty(0,\ell^{(n)})$ such that $\ol{\lambda_1}(J,g^{(n)},\hat\rho^{(n)})\to\infty$ as $n\to \infty$. Observe that we now can consider $(g^{(n)},\hat\rho^{(n)})$ as an element in $\mathcal{A}_+(I)$, so that we have
    \begin{equation}
        \ol{\lambda_1}(I,g^{(n)},\hat\rho^{(n)})=\ol{\lambda_1}(J,g^{(n)},\hat\rho^{(n)})\to\infty,
    \end{equation}
    as $n\to \infty$.
\end{proof}

\subsection{A lower bound for the smallest positive eigenvalue} 
We prove a stronger inequality, which implies \eqref{lower_bound_eigenvalue_1} directly.

\begin{proposition}\label{lower_bound_result}
    Let $G$ be a finite compact connected graph. Then, we have
    \begin{equation}\label{lower_bound_result_1}
        \ol{\lambda_1}(g,\rho)> \frac{L(G,g)}{\diam(G,g)},
    \end{equation}
    for every $(g,\rho)\in\mathcal{A}_+(G)$, where $\diam(G,g)$ is the diameter of $(G,g)$.
\end{proposition}

\begin{proof}
    Let $f$ be an eigenfunction of $\lambda_1(g,\rho)$. Since $\int_{G_0}f\rho\df x_g=0$, there exists $x_0\in (G,g)$ such that $f(x_0)=0$. Let $y\in (G,g)$ such that $|f(y)|=\|f\|_{L^\infty(G,g)}$ and $P$ be the shortest path connecting $x_0$ and $y$. Then, we have
    \begin{equation}
        \begin{aligned}
            \|f\|_{L^{\infty}(G,g)}^2&=|f(y)|^2=\left|\int_{P}f'(x)\df x_g\right|^2\le L(P,g)\int_{P}(f'(x))^2\df x_g\le \diam(G,g)\int_{G_0} |\nabla_g f|_g^2\df x_g
        \end{aligned}
    \end{equation}

    Therefore,
    \begin{equation}
        \begin{aligned}
            \ol{\lambda_1}(g,\rho)=\frac{\int_{G_0}|\nabla_g f|_g^2\df x_g\cdot \int_{G_0}\rho\df x_gL(G,g)}{\int_{G_0} f^2\rho\df x_g}\ge \frac{\int_{G_0}|\nabla_g f|_g^2\df x_g\cdot \int_{G_0}\rho\df x_gL(G,g)}{\|f\|_{L^\infty(G,g)}^2\int_{G_0}\rho \df x_g}\ge\frac{L(G,g)}{\diam(G,g)}
        \end{aligned}
    \end{equation}
    
    Equality of inequality \eqref{lower_bound_result_1} holds if and only if $f$ is a constant function. However, this contradicts the fact that constant functions cannot be $\lambda_1$-eigenfunctions. Thus, inequality \eqref{lower_bound_result_1} is strict.
\end{proof}

\begin{remark}
    We note that for the case $\rho\equiv 1$, if $G$ is an interval, then inequality \eqref{lower_bound_result_1} is weaker than inequality \eqref{lower_simple_eigenvalue}, while if $(G,g)$ has small diameter, then inequality \eqref{lower_bound_result_1} is stronger than inequality \eqref{lower_simple_eigenvalue}. 
\end{remark}

\section{Example of extremal metrics and geodesic nets}
In this section, we look at some examples of extremal metrics for the functional $\ol{\lambda}_1(\cdot)$ and find their corresponding geodesic nets on spheres. Recall that the necessary conditions for a smooth curve to be a geodesic on a sphere follow from the 1D version of Takahashi's Theorem (the general version can be found in \cite[Theorem 3]{Takahashi}).

\begin{theorem}
    Let $I$ be a compact interval, $n\in\mbb{N}$ and $\gamma:I\to\mbb{R}^n$ be a smooth curve. Suppose further that $\gamma$ is parametrised by arclength. Then, $\gamma$ is an isometric immersion to $\mbb{S}^{n-1}_{R}$ if and only if $-\gamma''=R^{-1}\gamma$.
\end{theorem}

\subsection{Graphs with pendants}
Pendants---also called leaves---of graphs are vertices of degree one. We observe that for a metric graph $(G,g)$ with at least one pendant, it is impossible to immerse $(G,g)$ into any sphere via an isometric minimal immersion, since the unit tangent vector at the pendant on the net does not vanish. 

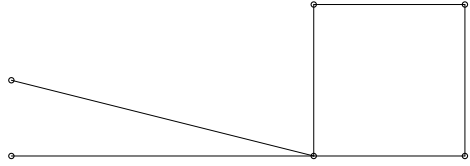
\begin{figure}[h]
\begin{tikzpicture}[scale=1]
	\tkzSetUpPoint[fill=white]
	\tkzDefPoint(0,0){v1}
	\tkzDefPoint(2,0){v2}
	\tkzDefPoint(2,2){v3}
	\tkzDefPoint(0,2){v4}
	\tkzDefPoint(-4,0){u1}
	\tkzDefPoint(-4,1){u2}

	\tkzDrawPoints(v1,v2,v3,v4,u1,u2)
	\tkzDrawSegments(v1,v2 v2,v3 v3,v4 v4,v1 u1,v1 u2,v1)
\end{tikzpicture}
    \caption{A graph with pendants}
\end{figure}

Therefore, there are no extremal pairs for the eigenvalue problem with a density function. Moreover, if $g_*$ is an extremal metric for $\ol{\lambda_k}(\cdot)$, then clearly $4\lambda_k(g_*)$ is an eigenvalue, according to Theorem \ref{main_theorem1} and Proposition \ref{4lambda_theorem}.

\subsection{Pumpkin graphs}

A pumpkin graph---also called mandarin or watermelon---is a graph with two vertices and no loop; i.e. all edges have these two vertices as the endpoints (see Figure \ref{mandarin}). For pumpkin graphs with two edges, we have $\ol{\lambda_1}(g)=4\pi^2$ for all metrics $g$. Thus, all metrics are extremal, and we can easily immerse the metric graph into any circle. 

For any pumpkin graph $G$ with at least $3$ edges, by \cite[Theorem 4.2]{Kennedy16}, regular metrics are unique maximisers for the functional $\ol{\lambda_1}(\cdot)$. Let $g=\pi^2 g_*$, then a simple calculation shows that $\lambda_1(g)=1$. We now show that it is possible to immerse $(G,g)$ into $\mbb{S}^2$ by constructing a geodesic net on $\mbb{S}^2$ as follows:
\begin{enumerate}
    \item Let $N:=(0,0,1)$ and $S:=(0,0,-1)$.
    \item On the circle $\mbb{S}^2\cap\{z=0\}$, consider a regular polygon $P_1,P_2,\dots,P_{|E|}$.
    \item For each $i\in\{1,\dots,|E|\}$, construct a geodesic $\gamma_i$ starting on $N$, passing through $P_i$ and ending at $S$.
\end{enumerate}

Then, the geodesic net $\Gamma$ formed by $\{\gamma_i\}_{i=1}^{|E|}$ and $\mathcal{V}=\{N,S\}$ is an immersion of $(G,g)$ on $\mbb{S}^2$ via some isometric maps constructed by $\lambda_1(g)$-eigenfunctions.

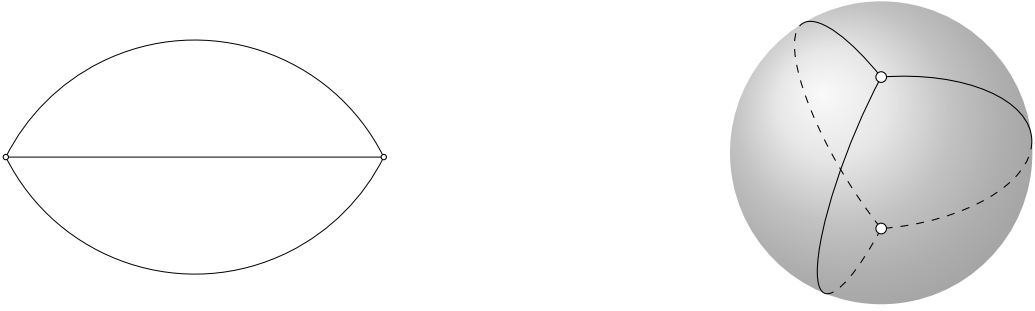
\begin{figure}[h]
    \centering
    \begin{subfigure}[b]{0.45\textwidth}
    	\centering
         \begin{tikzpicture}[scale=5]
	\tkzSetUpPoint[fill=white]
	
	\tkzDefPoint(0,0){o}
	\tkzDefPoint(1,0){o1}
	\tkzDefMidPoint(o,o1)\tkzGetPoint{m}
	\tkzDefPoint(0.5,0.25){c1}
	\tkzDefPoint(0.5,-0.25){c2}
	\tkzDrawArc(c1,o)(o1)
	\tkzDrawArc(c2,o1)(o)
	\tkzDefMidArc(c1,o,o1)\tkzGetPoint{m1}
	\tkzDefMidArc(c2,o1,o)\tkzGetPoint{m2}

	\tkzDrawSegments(o,o1)
	\tkzDrawPoints(o,o1)
	
\end{tikzpicture}
    \end{subfigure}
    \hfill
    \begin{subfigure}[b]{0.45\textwidth}
    	\centering
    \begin{tikzpicture}
	[scale=2,
		tdplot_main_coords]

	\coordinate (O) at (0,0,0);
 \shade[ball color = lightgray,
    opacity = 0.5
] (0,0,0) circle (1cm);

 	\tdplotsetthetaplanecoords{0}
		
		\tdplotdrawarc[tdplot_rotated_coords]{(O)}{1}{0}{120}{}{}
   	\tdplotsetthetaplanecoords{0}
		
		\tdplotdrawarc[dashed,tdplot_rotated_coords]{(O)}{1}{120}{180}{}{}

	\tdplotsetthetaplanecoords{120}
		
		\tdplotdrawarc[tdplot_rotated_coords]{(O)}{1}{0}{90}{}{}
	\tdplotsetthetaplanecoords{120}
		
		\tdplotdrawarc[dashed,tdplot_rotated_coords]{(O)}{1}{90}{180}{}{}
  
  	\tdplotsetthetaplanecoords{240}
		
		\tdplotdrawarc[tdplot_rotated_coords]{(O)}{1}{0}{60}{}{}

  	\tdplotsetthetaplanecoords{240}
		
		\tdplotdrawarc[dashed,tdplot_rotated_coords]{(O)}{1}{60}{180}{}{}

\draw[fill = white!50] (0,0,1) circle (1pt);
\draw[fill = white!50] (0,0,-1) circle (1pt);

\end{tikzpicture}

\end{subfigure}
	    \caption{A regular pumpkin graph with three edges and its corresponding geodesic nets on $\mbb{S}^2$ via $\lambda_1$-eigenfunctions.}
	\label{mandarin}
\end{figure}

\subsection{Flower graphs}\label{flower}
Recall that a flower graph is a graph with one vertex and all edges are loops (see Figure \ref{flower_picture}). In \cite[Corollary 2.8]{Band_optimiser}, Band and Lévy proved that regular metrics are unique maximisers for the functional $\ol{\lambda_1}(\cdot)$. Observe that if $G$ is a single loop, then all metrics on $G$ are regular, and we can map $G$ to any sphere easily.

Now, we consider a flower graph $G$ with at least two edges. Let $g$ be a regular metric on $G$ and $\ell$ be the common edge length. Then, a simple calculation shows that $\lambda_1(g)=(\pi/\ell)^2$. We show that it is impossible to immerse $(G,g)$ into any sphere $\mbb{S}^m$ via an isometric minimal immersion constructed by $\lambda_1(g)$-eigenfunctions. 

Indeed, suppose that there exists an isometric immersion $F:(G,g)\to\mbb{S}^{m}$ for some $m\in\mbb{N}$ and $F=(f_1,\dots,f_{m+1})$ such that $f_1,\dots,f_{m+1}\in E(\lambda_1(g))$. Then, Takahashi's Theorem implies that $\lambda_1=1$ and $\ell=\pi$. Observe that for each edge $e\in E$, $F(e)$ must be a multiple of some complete geodesics on $\mbb{S}^m$. Hence $|F(e)|=2n\pi$ for some $n\in\mbb{N}$, so that $\ell=2n\pi$ since $F$ is an isometric immersion, which is a contradiction. 

\begin{corollary}\label{no_maximiser}
	Let $G$ be a flower graph with $|E|\ge 2$ and $\alpha\notin\{1/2,1\}$. Then the functional $\ol{\lambda_1}^{(\alpha)}(\cdot,\cdot)$ does not have any extremal pair. Consequently, $\ol{\lambda_1}^{(\alpha)}(\cdot,\cdot)$ does not have any minimiser.
	\end{corollary}

\subsection{Necklace graphs}
A necklace graph is a graph $G=(V,E)$ with $V=\{v_1,v_2,\dots,v_n\}$ and $|E|=2(n-1)$, such that for every adjacent pair of vertices $v_j,v_{j+1}$, there are exactly two edges connecting them. A metric $g$ of $G$ is called symmetric if every pair of parallel edges connecting two vertices have the same length. By \cite[Theorem 2.1]{Band_optimiser}, symmetric metrics are unique minimisers of the functional $\ol{\lambda_1}(\cdot)$ on necklace graphs. 

\begin{figure}[h]
	
	\centering
	\begin{tikzpicture}[scale=1]
	\tkzInit[xmin = -4.25, xmax=4.25, ymin=-2, ymax=2]
	\tkzSetUpPoint[fill=white]
	\tkzDefPoint(-4,0){v1}
	\tkzDefPoint(-1,0){v2}
	\tkzDefPoint(1,0){v3}
	\tkzDefPoint(5,0){v4}
	\tkzDefPoint(-2.5,0){u1}
	\tkzDefPoint(-0,0){u2}
	\tkzDefPoint(3,0){u3}
	\tkzDefPoint(-2.5,1.5){a1}
	\tkzDefPoint(-2.5,-1.5){b1}
	\tkzDefPoint(0,1){a2}
	\tkzDefPoint(0,-1){b2}
	\tkzDefPoint(3,2){a3}
	\tkzDefPoint(3,-2){b3}
	\tkzDrawCircles(u1,v1 u2,v2 u3,v3)

	\tkzDrawPoints(v1,v2,v3,v4)
	\tkzLabelPoint[above](a1){$\ell_1$}
	\tkzLabelPoint(b1){$\ell_1$}
	\tkzLabelPoint[above](a2){$\ell_2$}
	\tkzLabelPoint(b2){$\ell_2$}
	\tkzLabelPoint[above](a3){$\ell_3$}
	\tkzLabelPoint(b3){$\ell_3$}
	\tkzLabelPoint[below left](v1){$v_1$}
	\tkzLabelPoint[below left](v2){$v_2$}
	\tkzLabelPoint[below left](v3){$v_3$}
	\tkzLabelPoint[below left](v4){$v_4$}
\end{tikzpicture}
	\caption{A symmetric necklace with 4 vertices}
	
\end{figure}
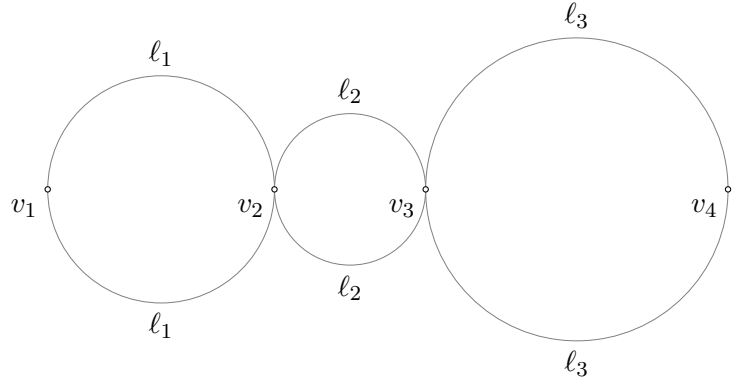

Suppose that $|E|\ge 3$ and let $g$ be a symmetric metric on $G$ with length $\ell_j$ to be the length of edges connecting $v_j$ and $v_{j+1}$. A simple calculation shows that:
\begin{equation}
    \lambda_1(g)=\left(\frac{\pi}{\ell_1+\cdots+\ell_{n-1}}\right)^2.
\end{equation}

We show that it is impossible to immerse $(G,g)$ into any sphere $\mbb{S}^m$ via an isometric minimal immersion constructed by $\lambda_1(g)$-eigenfunctions. Indeed, suppose that there exists an isometric immersion $F:(G,g)\to\mbb{S}^{m}$ for some $m\in\mbb{N}$ and $F=(f_1,\dots,f_{m+1})$ such that $f_1,\dots,f_{m+1}$ are $\lambda_1(g)$-eigenfunctions. Then, we have $\lambda_1(g)=1$ so that $\sum\ell_j=\pi$. Let $e_j$ and $\tilde e_j$ be edges connecting $v_{j}$ and $v_{j+1}$ and let $\gamma_j=F(e_j),\tilde\gamma_j=F(\tilde e_j)$. Then, $\gamma_1$ and $\tilde\gamma_1$ are in the same plane since the sum of their tangent vectors at $F(v_1)$ vanishes. Therefore, $\gamma_1\cup\tilde{\gamma}_1$ is a multiple of $\mbb{S}^1$ so that there exists $c_1\in\mbb{N}$ such that $\ell_1=c_1\pi$, which is a contradiction.

Hence, we cannot map symmetric necklaces to geodesic nets via isometric minimal immersions. The following corollary follows directly from the unique minimiser for the functional $\ol{\lambda_1}(\cdot)$ on necklace graphs. 

\begin{corollary}
	Let $G$ be a necklace graph with $|V|\ge 3$, then there is no minimiser for the functional $\ol{\lambda_1}^{(\alpha)}(\cdot,\cdot)$ for all $\alpha\ne 1$. Moreover, if $\alpha\notin\{1/2,1\}$, then there is no extremal pair for the functional $\ol{\lambda_1}^{(\alpha)}(\cdot,\cdot)$.
	\end{corollary}



\printbibliography

@article{ELSOUFI_lambda1,
    AUTHOR = {El Soufi, Ahmad and Ilias, Sa\"id},
     TITLE = {Riemannian manifolds admitting isometric immersions by their
              first eigenfunctions},
   JOURNAL = {Pacific J. Math.},
  FJOURNAL = {Pacific Journal of Mathematics},
    VOLUME = {195},
      YEAR = {2000},
    NUMBER = {1},
     PAGES = {91--99},
      ISSN = {0030-8730,1945-5844},
   MRCLASS = {53C42 (58E11 58J50)},
  MRNUMBER = {1781616},
MRREVIEWER = {Fabio\ Podest\`a},
       DOI = {10.2140/pjm.2000.195.91},
       URL = {https://doi.org/10.2140/pjm.2000.195.91},
}

@article{Nadirashvili1996,
	 AUTHOR = {Nadirashvili, N.},
     TITLE = {Berger's isoperimetric problem and minimal immersions of
              surfaces},
   JOURNAL = {Geom. Funct. Anal.},
  FJOURNAL = {Geometric and Functional Analysis},
    VOLUME = {6},
      YEAR = {1996},
    NUMBER = {5},
     PAGES = {877--897},
      ISSN = {1016-443X,1420-8970},
   MRCLASS = {53C42 (58G25)},
  MRNUMBER = {1415764},
MRREVIEWER = {S.\ Nik\v cevi\'c},
       DOI = {10.1007/BF02246788},
       URL = {https://doi.org/10.1007/BF02246788},
}

@article{ELSOUFI200889,
  AUTHOR = {El Soufi, Ahmad and Ilias, Sa\"id},
     TITLE = {Laplacian eigenvalue functionals and metric deformations on
              compact manifolds},
   JOURNAL = {J. Geom. Phys.},
  FJOURNAL = {Journal of Geometry and Physics},
    VOLUME = {58},
      YEAR = {2008},
    NUMBER = {1},
     PAGES = {89--104},
      ISSN = {0393-0440,1879-1662},
   MRCLASS = {58J50 (35P05 58E11)},
  MRNUMBER = {2378458},
MRREVIEWER = {De\ Tang\ Zhou},
       DOI = {10.1016/j.geomphys.2007.09.008},
       URL = {https://doi.org/10.1016/j.geomphys.2007.09.008},
}

@article{Band_optimiser,
 AUTHOR = {Band, Ram and L\'evy, Guillaume},
     TITLE = {Quantum graphs which optimize the spectral gap},
   JOURNAL = {Ann. Henri Poincar\'e},
  FJOURNAL = {Annales Henri Poincar\'e. A Journal of Theoretical and
              Mathematical Physics},
    VOLUME = {18},
      YEAR = {2017},
    NUMBER = {10},
     PAGES = {3269--3323},
      ISSN = {1424-0637,1424-0661},
   MRCLASS = {05C50 (81Q35)},
  MRNUMBER = {3697195},
       DOI = {10.1007/s00023-017-0601-2},
       URL = {https://doi.org/10.1007/s00023-017-0601-2},}

@article{relaxing,
 AUTHOR = {Karpukhin, Mikhail and M\'etras, Antoine},
     TITLE = {Laplace and {S}teklov extremal metrics via {$n$}-harmonic
              maps},
   JOURNAL = {J. Geom. Anal.},
  FJOURNAL = {Journal of Geometric Analysis},
    VOLUME = {32},
      YEAR = {2022},
    NUMBER = {5},
     PAGES = {Paper No. 154, 36},
      ISSN = {1050-6926,1559-002X},
   MRCLASS = {58C40 (58E11 58E20)},
  MRNUMBER = {4386422},
       DOI = {10.1007/s12220-022-00891-6},
       URL = {https://doi.org/10.1007/s12220-022-00891-6},
}

@article{surface_case,
 AUTHOR = {Fraser, Ailana and Schoen, Richard},
     TITLE = {Minimal surfaces and eigenvalue problems},
 BOOKTITLE = {Geometric analysis, mathematical relativity, and nonlinear
              partial differential equations},
    SERIES = {Contemp. Math.},
    VOLUME = {599},
     PAGES = {105--121},
 PUBLISHER = {Amer. Math. Soc., Providence, RI},
      YEAR = {2013},
      ISBN = {978-0-8218-9149-0},
   MRCLASS = {35P15 (35R01 53A10 58E12 58J50)},
  MRNUMBER = {3202476},
MRREVIEWER = {Anna\ Maria\ Candela},
       DOI = {10.1090/conm/599/11927},
       URL = {https://doi.org/10.1090/conm/599/11927},
}

@article{Takahashi,
    AUTHOR = {Takahashi, Tsunero},
     TITLE = {Minimal immersions of {R}iemannian manifolds},
   JOURNAL = {J. Math. Soc. Japan},
  FJOURNAL = {Journal of the Mathematical Society of Japan},
    VOLUME = {18},
      YEAR = {1966},
     PAGES = {380--385},
      ISSN = {0025-5645,1881-1167},
   MRCLASS = {53.74},
  MRNUMBER = {198393},
MRREVIEWER = {J.\ Eells},
       DOI = {10.2969/jmsj/01840380},
       URL = {https://doi.org/10.2969/jmsj/01840380},
}

@article {Kennedy16,
    AUTHOR = {Kennedy, James B. and Kurasov, Pavel and Malenov\'a, Gabriela
              and Mugnolo, Delio},
     TITLE = {On the spectral gap of a quantum graph},
   JOURNAL = {Ann. Henri Poincar\'e},
  FJOURNAL = {Annales Henri Poincar\'e. A Journal of Theoretical and
              Mathematical Physics},
    VOLUME = {17},
      YEAR = {2016},
    NUMBER = {9},
     PAGES = {2439--2473},
      ISSN = {1424-0637,1424-0661},
   MRCLASS = {35R02 (35P05 35P15)},
  MRNUMBER = {3535868},
       DOI = {10.1007/s00023-016-0460-2},
       URL = {https://doi.org/10.1007/s00023-016-0460-2},
}

@misc{ariturk2016,
      title={Eigenvalue estimates on quantum graphs}, 
      author={Sinan Ariturk},
      year={2016},
      eprint={1609.07471},
      archivePrefix={arXiv},
      primaryClass={math.SP},
      url={https://arxiv.org/abs/1609.07471}, 
}

@article{Berkolaiko_2017,
 AUTHOR = {Berkolaiko, Gregory and Kennedy, James B. and Kurasov, Pavel
              and Mugnolo, Delio},
     TITLE = {Edge connectivity and the spectral gap of combinatorial and
              quantum graphs},
   JOURNAL = {J. Phys. A},
  FJOURNAL = {Journal of Physics. A. Mathematical and Theoretical},
    VOLUME = {50},
      YEAR = {2017},
    NUMBER = {36},
     PAGES = {365201, 29},
      ISSN = {1751-8113,1751-8121},
   MRCLASS = {81Q35},
  MRNUMBER = {3688110},
MRREVIEWER = {Peter\ N.\ Zhevandrov},
       DOI = {10.1088/1751-8121/aa8125},
       URL = {https://doi.org/10.1088/1751-8121/aa8125},
}

@article {MR4311579,
    AUTHOR = {Girouard, Alexandre and Karpukhin, Mikhail and Lagac\'e, Jean},
     TITLE = {Continuity of eigenvalues and shape optimisation for {L}aplace
              and {S}teklov problems},
   JOURNAL = {Geom. Funct. Anal.},
  FJOURNAL = {Geometric and Functional Analysis},
    VOLUME = {31},
      YEAR = {2021},
    NUMBER = {3},
     PAGES = {513--561},
      ISSN = {1016-443X,1420-8970},
   MRCLASS = {49R05 (35P05 46E35 49Q10)},
  MRNUMBER = {4311579},
MRREVIEWER = {Andrzej\ M.\ My\'sli\'nski},
       DOI = {10.1007/s00039-021-00573-5},
       URL = {https://doi.org/10.1007/s00039-021-00573-5},
}

@article {lower_bound_1,
    AUTHOR = {Nicaise, Serge},
     TITLE = {Spectre des r\'eseaux topologiques finis},
   JOURNAL = {Bull. Sci. Math. (2)},
  FJOURNAL = {Bulletin des Sciences Math\'ematiques. 2e S\'erie},
    VOLUME = {111},
      YEAR = {1987},
    NUMBER = {4},
     PAGES = {401--413},
      ISSN = {0007-4497},
   MRCLASS = {58G25 (35P15)},
  MRNUMBER = {921561},
MRREVIEWER = {Monique\ Dauge},
}

@article {lower_bound_2,
    AUTHOR = {Friedlander, Leonid},
     TITLE = {Extremal properties of eigenvalues for a metric graph},
   JOURNAL = {Ann. Inst. Fourier (Grenoble)},
  FJOURNAL = {Universit\'e{} de Grenoble. Annales de l'Institut Fourier},
    VOLUME = {55},
      YEAR = {2005},
    NUMBER = {1},
     PAGES = {199--211},
      ISSN = {0373-0956,1777-5310},
   MRCLASS = {34L15 (34B45 58J50)},
  MRNUMBER = {2141695},
MRREVIEWER = {Xiangjin\ Xu},
       DOI = {10.5802/aif.2095},
       URL = {https://doi.org/10.5802/aif.2095},
}

@article {lower_bound_3,
    AUTHOR = {Kurasov, Pavel and Naboko, Sergey},
     TITLE = {Rayleigh estimates for differential operators on graphs},
   JOURNAL = {J. Spectr. Theory},
  FJOURNAL = {Journal of Spectral Theory},
    VOLUME = {4},
      YEAR = {2014},
    NUMBER = {2},
     PAGES = {211--219},
      ISSN = {1664-039X,1664-0403},
   MRCLASS = {34B45 (34L15 35P25)},
  MRNUMBER = {3232809},
MRREVIEWER = {Jun-Min\ Wang},
       DOI = {10.4171/JST/67},
       URL = {https://doi.org/10.4171/JST/67},
}

@book {intro_quantum_graph,
    AUTHOR = {Berkolaiko, Gregory and Kuchment, Peter},
     TITLE = {Introduction to quantum graphs},
    SERIES = {Mathematical Surveys and Monographs},
    VOLUME = {186},
 PUBLISHER = {American Mathematical Society, Providence, RI},
      YEAR = {2013},
     PAGES = {xiv+270},
      ISBN = {978-0-8218-9211-4},
   MRCLASS = {81Q35 (05C90 31C20 34B24 34B45 81Q50)},
  MRNUMBER = {3013208},
MRREVIEWER = {Delio\ Mugnolo},
       DOI = {10.1090/surv/186},
       URL = {https://doi.org/10.1090/surv/186},
}
\end{document}